\documentclass[submission,copyright,creativecommons]{eptcs}

\usepackage{url}

\providecommand{\event}{ACT 2024}

\usepackage{mvr}
\usepackage{breakurl}

\usepackage{tikz-cd}

\usepackage{tikzit}

\tikzstyle{morphism}=[fill=white, draw=black, shape=circle, inner sep=0pt, text width=15pt, align=center]
\tikzstyle{copy}=[fill=black, draw=black, shape=circle, inner sep=0pt, minimum size=5pt]

\tikzstyle{midarrow}=[-, midarrow]
\tikzstyle{dashedarrow}=[-, dashed, midarrow, draw=black]
\tikzstyle{midarrowdashed}=[-, dashed, midarrow, draw=black]
\tikzstyle{onlydashed}=[-, dashed]

\input{string.tikzdefs}

\newcommand{\C}{\mathscr{C}}

\newcommand{\D}{\mathscr{D}}

\newcommand{\SymMon}{\mathbf{SymMon}}
\newcommand{\Comp}{\mathbf{Comp}}
\DeclareMathOperator{\Core}{Core}

\newcommand{\Optic}{\mathbf{Optic}}
\newcommand{\Lens}{\mathbf{Lens}}
\newcommand{\Para}{\mathbf{Para}}
\newcommand{\Learn}{\mathbf{Learn}}
\newcommand{\IntLearn}{\mathbf{IntLearn}}

\newcommand{\Atemp}{\mathbf{Atemp}}

\newcommand{\hto}{\ensuremath{\,\mathaccent\shortmid\rightarrow\,}}

\title{Learners are Almost Free Compact Closed}
\author{Mitchell Riley
\institute{Mathematics, Division of Science}
\institute{New York University Abu Dhabi}
\email{mitchell.v.riley@nyu.edu}
\thanks{
The author acknowledges support by {\it Tamkeen} under {\it NYUAD Research Institute grant} {\tt CG008}.
}
}
\def\titlerunning{Learners are Almost Free Compact Closed}
\def\authorrunning{M. Riley}
\begin{document}
\maketitle

\begin{abstract}
  The category of learners has a pleasant symmetric formulation when
  the morphisms are considered up to a coarser equivalence than the
  one originally described in the paper "Backprop as Functor". A
  quotient of this modified category gives a new construction of the
  free compact closed category on a symmetric monoidal category.
\end{abstract}

In this note I expand on an observation made at the end
of~\cite{riley:optics} and further discussed
online~\cite{riley:zulip-involution}. If the definition of a
learner~\cite{fst:backprop-as-functor} is tweaked slightly so that
learners are considered identical under a coarser equivalence relation
than before, a striking symmetry in the definition is revealed.

This new category comes quite close to being compact closed. There is
a strict symmetric monoidal involution and each object has a cup and
cap morphism connecting it to its ``dual'' under this involution.
However, the snake equations fail and so these are not true duals.
Quotienting the category to force the equations to hold results in a
compact closed category, and we end the note by showing that it is in
fact the \emph{free} compact closed category on the starting symmetric
monoidal category.

Compact closed categories support a more powerful diagrammatic
calculus than bare symmetric monoidal categories, so this work may
make it easier to design and reason about learners compositionally.
The operation that dualises a learner does not appear to have been
previously noticed and there is much to investigate. As asked by David
Spivak, what is the dual of a neuron?

\section{Extensional Learners}

Learners are a categorical construction used in a compositional
approach to supervised learning and
backpropagation~\cite{fst:backprop-as-functor,
  cggwz:categorical-learning}. We first recall their definition.

\begin{definition}[{\cite[Definition 2.1]{fst:backprop-as-functor}}]
  For sets $A$ and $B$, a \emph{learner} $A \hto B$ is a ``parameter'' set $P$
  together with a triple of functions $I : P \times A \to B$ and
  $U : P \times A \times B \to P$ and $r : P \times A \times B \to A$,
  considered up to isomorphism of parameter sets. That is, two learners
  $(P, I, U, r)$ and $(P', I', U', r')$ are identified whenever there
  is a bijection $f : P \to P'$ with
\begin{alignat*}{2}
&I'(f(p), a) &&= I(p, a) \\
&U'(f(p), a, b) &&= f(U(p, a, b)) \\
&r'(f(p), a, b) &&= r(p, a, b)
\end{alignat*}
\end{definition}

Learners expose their parameter sets $P$ to the outside world (up to isomorphism), and two
learners with non-isomorphic parameter sets are considered different
even if they `behave the same' on all input data $A$ and $B$. Cribbing
some terminology from type theory, we will label these learners
\emph{intensional}: their identity depends on the specific choice of
parameter set even if it is not observable through their input-output
behaviour. (Unfortunately this clashes with the meaning of
intensional/extensional as used in
\cite{hefford-comfort:coend-quantum-combs}.)

The equivalence relation on learners has a clear generalisation to any
category with finite products, and in this generality:

\begin{proposition}[{\cite[Proposition 2.4]{fst:backprop-as-functor}}]
  Intensional learners in any category with finite products form a
  symmetric monoidal category $\IntLearn_\C$. \qed{}
\end{proposition}

The data underlying a learner $A \hto B$ with fixed parameter set $P$
has an appealing symmetric description, via the chain of isomorphisms
\begin{align*}
&\C(P \times A, B) \times \C(P \times A \times B, P) \times \C(P \times A \times B, A) \\
\cong{} &\C(P \times A, B) \times \C(P \times A \times B, P \times A) \\
\cong{} &\int^{Q : \C} \C(P \times A, Q) \times \C(P \times A, B) \times \C(Q \times B, P \times A) \\
\cong{} &\int^{Q : \C} \C(P \times A, Q \times B) \times \C(Q \times B, P \times A)
\end{align*}
The quotient up to isomorphism of $P$ can then be expressed as an
additional coend over the core of $\C$:
\begin{align*}
  \IntLearn(A, B) = \int^{P : \Core(\C)} \int^{Q : \C} \C(P \times A, Q \times B) \times \C(Q \times B, P \times A)
\end{align*}
Three generalisations immediately suggest themselves:
\begin{itemize}
\item replace the former coend with one over all of $\C$ rather than
  only the isomorphisms,
\item replace $\times$ with an arbitrary symmetric monoidal product
  $\otimes$, and,
\item allow the repeat occurrences of $A$ and $B$ to be different
  objects, as in the category of optics~\cite{profunctor:update,
    riley:optics}.
\end{itemize}
We arrive at our generalised definition of learner.
\begin{definition}
  For objects $A, A', B, B'$ of a symmetric monoidal category $\C$, an
  \emph{extensional learner} $(A, A') \hto (B, B')$ is an element of
  \begin{align*}
    \Learn_\C((A, A'), (B, B')) := \int^{P, Q : \C} \C(P \otimes A, Q \otimes B) \times \C(Q \otimes B', P \otimes A')
  \end{align*}
  The extensional learner represented by a pair of maps
  $f : P \otimes A \to Q \otimes B$ and
  $g : Q \otimes B' \to P \otimes A'$ will be written
  $(f \mid g) : (A, A') \hto (B, B')$.
\end{definition}
Because these are our focus, we omit `extensional' and call these
simply \emph{learners}.

When specialised back to the case of sets and binary products, this
definition unwinds to a coarser notion of equivalence between a pair
of intensional learners. Two learners $(P, I, U, r)$ and
$(P', I', U', r')$ are identified by this new definition when there is
a (not-necessarily-bijective) function $f : P \to P'$ and an
intermediate function $\hat{U} : P' \times A \times B \to P$ such that
\begin{alignat*}{2}
  &I'(f(p), a) &&= I(p, a) \\
  &r'(f(p), a, b) &&= r(p, a, b) \\
  &\hat{U}(f(p), a, b) &&= U(p, a, b) \\
  &U'(p', a, b) &&= f(\hat{U}(p', a, b))
\end{alignat*}

That is, the relationship of $f$ to $I$ and $r$ is the same as before,
but for $U$ we require a diagonal filler in the diagram
\begin{center}
  \begin{tikzcd}
    P \times A \times B \ar[r, "U"] \ar[d, "f \times A \times B" swap] & P \ar[d, "f"] \\
    P' \times A \times B \ar[r, "{U'}" swap] \ar[ur, dashed, "{\hat{U}}"] & P'
  \end{tikzcd}
\end{center}

\begin{lemma}\label{lem:int-are-ext}
  Intensionally equivalent learners are extensionally equivalent.
\end{lemma}
\begin{proof}
  The bijection $f : P \to P'$ satisfies the first two equations of
  extensional equivalence by assumption. For the intermediate function
  $\hat{U} : P' \times A \times B \to P$, take
  $\hat{U}(p', a, b) := U(f^{-1}(p'), a, b)$, then
  \begin{align*}
    \hat{U}(f(p), a, b) &= U(f^{-1}(f(p)), a, b) = U(p, a, b) \\
    f(\hat{U}(p', a, b)) &= f(U(f^{-1}(p'), a, b)) = U'(f(f^{-1}(p')), a, b) = U'(p', a, b)
  \end{align*}
\end{proof}

\begin{remark}
  The extensional relation is similar to the notion of 2-morphism between
  learners given in \cite[Definition 7.1]{fst:backprop-as-functor},
  which only requires that the outer square in the above diagram
  commutes. The intensional relation identifies two learners when
  there is a 2-isomorphism between them.

  In~\cite{fong-johnson:lenses-learners}, Fong and Johnson consider
  two intensional learners identical when there is a 2-morphism such
  that $f$ is only \emph{surjective}, and not necessarily a bijection.
  This is a coarser notion of equivalence still: for example, it
  identifies all intensional learners $1 \hto 1$.
\end{remark}

\begin{proposition}
  Learners form a category $\Learn_\C$, where the identity learner
  $(A, A') \hto (A, A')$ is
  $(\id_{I \otimes A} \mid \id_{I \otimes A'})$, and the composite of
  $(l_1 \mid r_1) : (A, A') \hto (B, B')$ and
  $(l_2 \mid r_2) : (B, B') \hto (C, C')$ with representatives
  \begin{align*}
    l_1 &: P_1 \otimes A \to Q_1 \otimes B \\
    r_1 &: Q_1 \otimes B' \to P_1 \otimes A' \\
    l_2 &: P_2 \otimes B \to Q_2 \otimes C \\
    r_2 &: Q_2 \otimes C' \to P_2 \otimes B'
  \end{align*}
  is given by
  \begin{center}
    \begin{tikzpicture}[scale=2]
	\begin{pgfonlayer}{nodelayer}
		\node [style=none, label={left:$P_2$}] (9) at (-1.5, 3) {};
		\node [style=none, label={left:$P_1$}] (10) at (-1.5, 2) {};
		\node [style=none, label={left:$A$}] (12) at (-1.5, 1) {};
		\node [style=none] (13) at (0, 3) {};
		\node [style=none] (15) at (0, 1) {};
		\node [style=none] (16) at (1.5, 3) {};
		\node [style=none] (18) at (1.5, 1) {};
		\node [style=morphism] (20) at (-0.75, 1.5) {$l_1$};
		\node [style=none, label={right:$Q_1$}] (21) at (3, 3) {};
		\node [style=none, label={right:$Q_2$}] (22) at (3, 2) {};
		\node [style=none, label={right:$C$}] (23) at (3, 1) {};
		\node [style=morphism] (24) at (2.25, 1.5) {$l_2$};
		\node [style=none] (25) at (0, 2) {};
		\node [style=none] (26) at (1.5, 2) {};
	\end{pgfonlayer}
	\begin{pgfonlayer}{edgelayer}
		\draw [in=120, out=0, looseness=0.75] (10.center) to (20);
		\draw [in=240, out=0, looseness=0.75] (12.center) to (20);
		\draw [in=-180, out=-60, looseness=0.75] (20) to (15.center);
		\draw [in=240, out=0, looseness=0.75] (18.center) to (24);
		\draw [in=-180, out=60, looseness=0.75] (24) to (22.center);
		\draw [in=-180, out=-60, looseness=0.75] (24) to (23.center);
		\draw (9.center) to (13.center);
		\draw (16.center) to (21.center);
		\draw (15.center) to (18.center);
		\draw [in=0, out=120, looseness=0.75] (24) to (26.center);
		\draw [in=0, out=180, looseness=1.25] (26.center) to (13.center);
		\draw [in=-180, out=0, looseness=1.25] (25.center) to (16.center);
		\draw [in=-180, out=60, looseness=0.75] (20) to (25.center);
	\end{pgfonlayer}
\end{tikzpicture} \quad \begin{tikzpicture}[scale=2]
	\begin{pgfonlayer}{nodelayer}
		\node [style=none, label={left:$Q_1$}] (9) at (-1.5, 3) {};
		\node [style=none, label={left:$Q_2$}] (10) at (-1.5, 2) {};
		\node [style=none, label={left:$C'$}] (12) at (-1.5, 1) {};
		\node [style=none] (13) at (0, 3) {};
		\node [style=none] (15) at (0, 1) {};
		\node [style=none] (16) at (1.5, 3) {};
		\node [style=none] (18) at (1.5, 1) {};
		\node [style=morphism] (20) at (-0.75, 1.5) {$r_2$};
		\node [style=none, label={right:$P_2$}] (21) at (3, 3) {};
		\node [style=none, label={right:$P_1$}] (22) at (3, 2) {};
		\node [style=none, label={right:$A'$}] (23) at (3, 1) {};
		\node [style=morphism] (24) at (2.25, 1.5) {$r_1$};
		\node [style=none] (25) at (0, 2) {};
		\node [style=none] (26) at (1.5, 2) {};
	\end{pgfonlayer}
	\begin{pgfonlayer}{edgelayer}
		\draw [in=120, out=0, looseness=0.75] (10.center) to (20);
		\draw [in=240, out=0, looseness=0.75] (12.center) to (20);
		\draw [in=-180, out=-60, looseness=0.75] (20) to (15.center);
		\draw [in=240, out=0, looseness=0.75] (18.center) to (24);
		\draw [in=-180, out=60, looseness=0.75] (24) to (22.center);
		\draw [in=-180, out=-60, looseness=0.75] (24) to (23.center);
		\draw (9.center) to (13.center);
		\draw (16.center) to (21.center);
		\draw (15.center) to (18.center);
		\draw [in=0, out=120, looseness=0.75] (24) to (26.center);
		\draw [in=0, out=180, looseness=1.25] (26.center) to (13.center);
		\draw [in=-180, out=0, looseness=1.25] (25.center) to (16.center);
		\draw [in=-180, out=60, looseness=0.75] (20) to (25.center);
	\end{pgfonlayer}
\end{tikzpicture}.
  \end{center}
\end{proposition}
\begin{proof}
  Similar to \cite[Proposition 2.0.3]{riley:optics}. The unit laws for
  composition follow quickly from the unit laws in $\C$. Associativity
  of composition follows (for the left component) from the isotopy
  of the following string diagrams together with associativity of $\C$. 
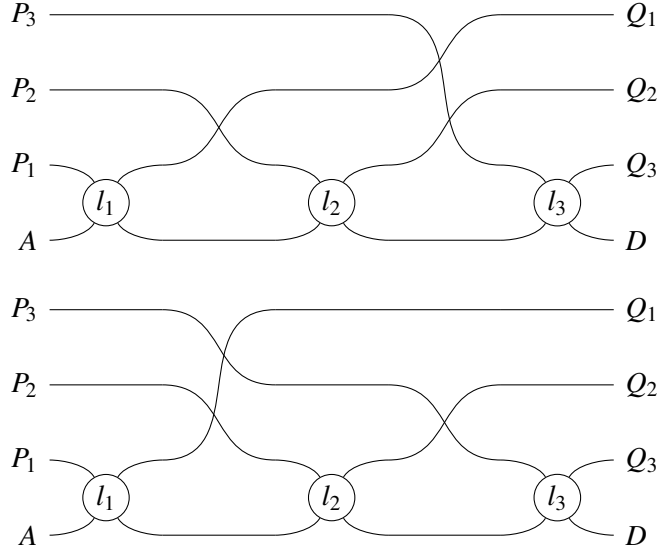
  \begin{center}
    \begin{tikzpicture}[scale=2]
	\begin{pgfonlayer}{nodelayer}
		\node [style=none, label={left:$P_2$}] (9) at (-5.75, 2.5) {};
		\node [style=none, label={left:$P_1$}] (10) at (-5.75, 1.5) {};
		\node [style=none, label={left:$A$}] (12) at (-5.75, 0.5) {};
		\node [style=none] (13) at (-4.25, 2.5) {};
		\node [style=none] (15) at (-4.25, 0.5) {};
		\node [style=none] (16) at (-2.75, 2.5) {};
		\node [style=none] (18) at (-2.75, 0.5) {};
		\node [style=morphism] (20) at (-5, 1) {$l_1$};
		\node [style=none] (23) at (-1.25, 0.5) {};
		\node [style=morphism] (24) at (-2, 1) {$l_2$};
		\node [style=none] (25) at (-4.25, 1.5) {};
		\node [style=none] (26) at (-2.75, 1.5) {};
		\node [style=none, label={left:$P_3$}] (41) at (-5.75, 3.5) {};
		\node [style=none] (43) at (0.25, 0.5) {};
		\node [style=none, label={right:$Q_3$}] (44) at (1.75, 1.5) {};
		\node [style=none, label={right:$D$}] (45) at (1.75, 0.5) {};
		\node [style=morphism] (46) at (1, 1) {$l_3$};
		\node [style=none] (48) at (-1.25, 3.5) {};
		\node [style=none] (50) at (-1.25, 2.5) {};
		\node [style=none] (55) at (0.25, 1.5) {};
		\node [style=none] (58) at (-1.25, 1.5) {};
		\node [style=none] (59) at (0.25, 2.5) {};
		\node [style=none, label={right:$Q_2$}] (60) at (1.75, 2.5) {};
		\node [style=none] (61) at (0.25, 3.5) {};
		\node [style=none, label={right:$Q_1$}] (62) at (1.75, 3.5) {};
	\end{pgfonlayer}
	\begin{pgfonlayer}{edgelayer}
		\draw [in=120, out=0, looseness=0.75] (10.center) to (20);
		\draw [in=240, out=0, looseness=0.75] (12.center) to (20);
		\draw [in=-180, out=-60, looseness=0.75] (20) to (15.center);
		\draw [in=240, out=0, looseness=0.75] (18.center) to (24);
		\draw [in=-180, out=-60, looseness=0.75] (24) to (23.center);
		\draw (9.center) to (13.center);
		\draw (15.center) to (18.center);
		\draw [in=0, out=120, looseness=0.75] (24) to (26.center);
		\draw [in=0, out=180, looseness=1.25] (26.center) to (13.center);
		\draw [in=-180, out=0, looseness=1.25] (25.center) to (16.center);
		\draw [in=-180, out=60, looseness=0.75] (20) to (25.center);
		\draw [in=240, out=0, looseness=0.75] (43.center) to (46);
		\draw [in=-180, out=60, looseness=0.75] (46) to (44.center);
		\draw [in=-180, out=-60, looseness=0.75] (46) to (45.center);
		\draw (41.center) to (48.center);
		\draw (16.center) to (50.center);
		\draw [in=0, out=120, looseness=0.75] (46) to (55.center);
		\draw (23.center) to (43.center);
		\draw (59.center) to (60.center);
		\draw (61.center) to (62.center);
		\draw [in=-180, out=0, looseness=1.25] (50.center) to (61.center);
		\draw [in=-180, out=0, looseness=1.25] (58.center) to (59.center);
		\draw [in=-180, out=60, looseness=0.75] (24) to (58.center);
		\draw [in=0, out=180, looseness=1.25] (55.center) to (48.center);
	\end{pgfonlayer}
\end{tikzpicture} \quad \begin{tikzpicture}[scale=2]
	\begin{pgfonlayer}{nodelayer}
		\node [style=none, label={left:$P_2$}] (9) at (-5.75, 2.5) {};
		\node [style=none, label={left:$P_1$}] (10) at (-5.75, 1.5) {};
		\node [style=none, label={left:$A$}] (12) at (-5.75, 0.5) {};
		\node [style=none] (13) at (-4.25, 2.5) {};
		\node [style=none] (15) at (-4.25, 0.5) {};
		\node [style=none] (16) at (-2.75, 3.5) {};
		\node [style=none] (18) at (-2.75, 0.5) {};
		\node [style=morphism] (20) at (-5, 1) {$l_1$};
		\node [style=none] (23) at (-1.25, 0.5) {};
		\node [style=morphism] (24) at (-2, 1) {$l_2$};
		\node [style=none] (25) at (-4.25, 1.5) {};
		\node [style=none] (26) at (-2.75, 1.5) {};
		\node [style=none, label={left:$P_3$}] (41) at (-5.75, 3.5) {};
		\node [style=none] (43) at (0.25, 0.5) {};
		\node [style=none, label={right:$Q_3$}] (44) at (1.75, 1.5) {};
		\node [style=none, label={right:$D$}] (45) at (1.75, 0.5) {};
		\node [style=morphism] (46) at (1, 1) {$l_3$};
		\node [style=none] (48) at (-1.25, 2.5) {};
		\node [style=none] (55) at (0.25, 1.5) {};
		\node [style=none] (58) at (-1.25, 1.5) {};
		\node [style=none] (59) at (0.25, 2.5) {};
		\node [style=none, label={right:$Q_2$}] (60) at (1.75, 2.5) {};
		\node [style=none] (61) at (0.25, 3.5) {};
		\node [style=none, label={right:$Q_1$}] (62) at (1.75, 3.5) {};
		\node [style=none] (64) at (-4.25, 3.5) {};
		\node [style=none] (66) at (-2.75, 2.5) {};
	\end{pgfonlayer}
	\begin{pgfonlayer}{edgelayer}
		\draw [in=120, out=0, looseness=0.75] (10.center) to (20);
		\draw [in=240, out=0, looseness=0.75] (12.center) to (20);
		\draw [in=-180, out=-60, looseness=0.75] (20) to (15.center);
		\draw [in=240, out=0, looseness=0.75] (18.center) to (24);
		\draw [in=-180, out=-60, looseness=0.75] (24) to (23.center);
		\draw (9.center) to (13.center);
		\draw (15.center) to (18.center);
		\draw [in=0, out=120, looseness=0.75] (24) to (26.center);
		\draw [in=0, out=180, looseness=1.25] (26.center) to (13.center);
		\draw [in=-180, out=0, looseness=1.25] (25.center) to (16.center);
		\draw [in=-180, out=60, looseness=0.75] (20) to (25.center);
		\draw [in=240, out=0, looseness=0.75] (43.center) to (46);
		\draw [in=-180, out=60, looseness=0.75] (46) to (44.center);
		\draw [in=-180, out=-60, looseness=0.75] (46) to (45.center);
		\draw [in=0, out=120, looseness=0.75] (46) to (55.center);
		\draw (23.center) to (43.center);
		\draw (59.center) to (60.center);
		\draw (61.center) to (62.center);
		\draw [in=-180, out=0, looseness=1.25] (58.center) to (59.center);
		\draw (41.center) to (64.center);
		\draw (66.center) to (48.center);
		\draw (16.center) to (61.center);
		\draw [in=0, out=180, looseness=1.25] (66.center) to (64.center);
		\draw [in=-180, out=60, looseness=0.75] (24) to (58.center);
		\draw [in=0, out=180, looseness=1.25] (55.center) to (48.center);
	\end{pgfonlayer}
\end{tikzpicture}
  \end{center}
  The right component is symmetric. The coend relation ensures that
  the associativity morphisms in $\C$ for the triples
  $(P_1, P_2, P_3)$ and $(Q_1, Q_2, Q_3)$ can be used without changing
  the identity of the learner, as each such morphism is paired with
  its inverse across the gap.
\end{proof}

\begin{remark}
  Intensional learners manifestly do not form a locally small
  category: any choice of parameter set $P$ determines a learner
  $(1, 1) \hto (1, 1)$ with $U$ given by projection, and any two sets
  with different cardinalities determine two intensionally
  inequivalent learners. The situation is not improved with the
  extensional relation: these learners remain unequal and $\Learn_\C$
  is not locally small. (Compare with ordinary lenses, which can be
  described in a similar coend style but always form a locally small
  category.) For a symmetric monoidal category, the collection of
  learners $(I, I) \hto (I, I)$ is the \emph{trace} of $\C$ in the
  sense of~\cite{faro:trace}, which can be difficult to calculate even
  in simple cases.
\end{remark}

The category $\Learn_\C$ accepts an obvious functor
$\iota : \C \to \Learn_\C$, sending $f : A \to B$ to
$(I \otimes f \mid \id_{I \otimes I}) : (A, I) \hto (B, I)$. And, it
also accepts a functor from $\Optic_\C$, which sends an optic
$\langle l \mid r \rangle : (A, A') \hto (B, B')$ to the learner with
components
\begin{align*}
  I \otimes A &\to A \xrightarrow{l} M \otimes B \\
  M \otimes B' &\xrightarrow{r} A' \to I \otimes A'
\end{align*}
The observation that lenses include into learners as those with
trivial parameter set was made in \cite{fst:backprop-as-functor}, and
this fact is made especially clear by the characterisation
$\IntLearn_\C \cong \Para(\Lens_\C)$ given in
\cite{cggwz:categorical-learning}.

\begin{proposition}
  $\Learn_\C$ is symmetric monoidal, with action on objects given by
  $(A, A') \otimes (B, B') := (A \otimes B, B' \otimes A')$ and action
  on morphisms
  \begin{align*}
   (l_L \mid r_L) \otimes (l_R \mid r_R) : (A_L, A_L') \otimes (A_R,
    A_R') \to (B_L, B_L') \otimes (B_R, B_R')
  \end{align*}
  given by
  \begin{center}
    \begin{tikzpicture}[scale=2]
	\begin{pgfonlayer}{nodelayer}
		\node [style=none, label={right:$B_R$}] (42) at (2.5, -1.5) {};
		\node [style=none, label={left:$P_R$}] (45) at (-2.5, 0.5) {};
		\node [style=none, label={left:$P_L$}] (46) at (-2.5, 1.5) {};
		\node [style=none, label={left:$A_R$}] (47) at (-2.5, -1.5) {};
		\node [style=none, label={left:$A_L$}] (48) at (-2.5, -0.5) {};
		\node [style=none] (49) at (-1, -1.5) {};
		\node [style=none] (50) at (1, -1.5) {};
		\node [style=none, label={right:$Q_L$}] (51) at (2.5, 0.5) {};
		\node [style=none, label={right:$Q_R$}] (52) at (2.5, 1.5) {};
		\node [style=none] (54) at (-2.5, 0.5) {};
		\node [style=none] (55) at (-1, 0.5) {};
		\node [style=none] (56) at (-2.5, -0.5) {};
		\node [style=morphism] (57) at (0, -1) {$l_R$};
		\node [style=none] (58) at (-1, -0.5) {};
		\node [style=none] (59) at (-1, 1.5) {};
		\node [style=morphism] (60) at (0, 1) {$l_L$};
		\node [style=none] (61) at (1, 0.5) {};
		\node [style=none] (62) at (1, -0.5) {};
		\node [style=none, label={right:$B_L$}] (63) at (2.5, -0.5) {};
		\node [style=none] (64) at (1, 1.5) {};
	\end{pgfonlayer}
	\begin{pgfonlayer}{edgelayer}
		\draw (47.center) to (49.center);
		\draw [in=-180, out=0, looseness=1.25] (56.center) to (55.center);
		\draw [in=0, out=150] (57) to (58.center);
		\draw [in=180, out=-30] (57) to (50.center);
		\draw [in=0, out=-150] (57) to (49.center);
		\draw [in=0, out=180, looseness=1.25] (58.center) to (54.center);
		\draw [in=0, out=150] (60) to (59.center);
		\draw [in=-150, out=0] (55.center) to (60);
		\draw (46.center) to (59.center);
		\draw [in=150, out=0] (59.center) to (60);
		\draw [in=0, out=180, looseness=1.25] (63.center) to (61.center);
		\draw [in=-180, out=30] (57) to (62.center);
		\draw [in=30, out=180] (64.center) to (60);
		\draw [in=-30, out=180] (61.center) to (60);
		\draw [in=-180, out=30] (60) to (64.center);
		\draw [in=-180, out=0, looseness=1.25] (62.center) to (52.center);
		\draw [in=0, out=180, looseness=1.25] (51.center) to (64.center);
		\draw (50.center) to (42.center);
	\end{pgfonlayer}
\end{tikzpicture} \quad \begin{tikzpicture}[scale=2]
	\begin{pgfonlayer}{nodelayer}
		\node [style=none, label={right:$A'_L$}] (42) at (2.5, -1.5) {};
		\node [style=none, label={left:$Q_L$}] (45) at (-2.5, 0.5) {};
		\node [style=none, label={left:$Q_R$}] (46) at (-2.5, 1.5) {};
		\node [style=none, label={left:$B'_L$}] (47) at (-2.5, -1.5) {};
		\node [style=none, label={left:$B'_R$}] (48) at (-2.5, -0.5) {};
		\node [style=none] (49) at (-1, -1.5) {};
		\node [style=none] (50) at (1, -1.5) {};
		\node [style=none, label={right:$P_R$}] (51) at (2.5, 0.5) {};
		\node [style=none, label={right:$P_L$}] (52) at (2.5, 1.5) {};
		\node [style=none] (54) at (-2.5, 0.5) {};
		\node [style=none] (55) at (-1, 0.5) {};
		\node [style=none] (56) at (-2.5, -0.5) {};
		\node [style=morphism] (57) at (0, -1) {$r_L$};
		\node [style=none] (58) at (-1, -0.5) {};
		\node [style=none] (59) at (-1, 1.5) {};
		\node [style=morphism] (60) at (0, 1) {$r_R$};
		\node [style=none] (61) at (1, 0.5) {};
		\node [style=none] (62) at (1, -0.5) {};
		\node [style=none, label={right:$A'_R$}] (63) at (2.5, -0.5) {};
		\node [style=none] (64) at (1, 1.5) {};
	\end{pgfonlayer}
	\begin{pgfonlayer}{edgelayer}
		\draw (47.center) to (49.center);
		\draw [in=-180, out=0, looseness=1.25] (56.center) to (55.center);
		\draw [in=0, out=150] (57) to (58.center);
		\draw [in=180, out=-30] (57) to (50.center);
		\draw [in=0, out=-150] (57) to (49.center);
		\draw [in=0, out=180, looseness=1.25] (58.center) to (54.center);
		\draw [in=0, out=150] (60) to (59.center);
		\draw [in=-150, out=0] (55.center) to (60);
		\draw (46.center) to (59.center);
		\draw [in=150, out=0] (59.center) to (60);
		\draw [in=0, out=180, looseness=1.25] (63.center) to (61.center);
		\draw [in=-180, out=30] (57) to (62.center);
		\draw [in=30, out=180] (64.center) to (60);
		\draw [in=-30, out=180] (61.center) to (60);
		\draw [in=-180, out=30] (60) to (64.center);
		\draw [in=-180, out=0, looseness=1.25] (62.center) to (52.center);
		\draw [in=0, out=180, looseness=1.25] (51.center) to (64.center);
		\draw (50.center) to (42.center);
	\end{pgfonlayer}
\end{tikzpicture}
  \end{center}
\end{proposition}
\begin{proof}
  Similar to \cite[Theorem 2.0.12]{riley:optics}, but using a version
  of the `switched' tensor from \cite[Definition 2.1.1]{riley:optics}.
  The key facts to check are that the action on morphisms does not
  depend on the choice of representative and that $\otimes$ is
  functorial. The first is clear from the above diagrams: for example,
  a morphism on the $Q_L$ string between $l_L$ and $r_L$ can be slid from
  $l_L$ to $r_L$ in the string diagram (using the coend relation to
  hop the gap), and similarly for the three strings corresponding to $Q_R$, $P_L$ and $P_R$.

  Functoriality of $\otimes$ follows from the equality of the
  following diagrams, depicting the left component of the morphisms
  \begin{align*}
    ((l_{L2} \mid r_{L2}) \circ (l_{L1} \mid r_{L1}))
    &\otimes ((l_{R2} \mid r_{R2}) \circ (l_{R1} \mid r_{R1})) \\
  \intertext{and}
    ((l_{L2} \mid r_{L2}) \otimes (l_{R2} \mid r_{R2}))
    &\circ ((l_{L1} \mid r_{L1}) \otimes (l_{R1} \mid r_{R1}))
  \end{align*}
  respectively. (Imagine dragging the $l_{L1}$ node below the $P_{R2}$
  string to the left, and the $l_{L2}$ node below the $Q_{R1}$ string
  to the right.)
  
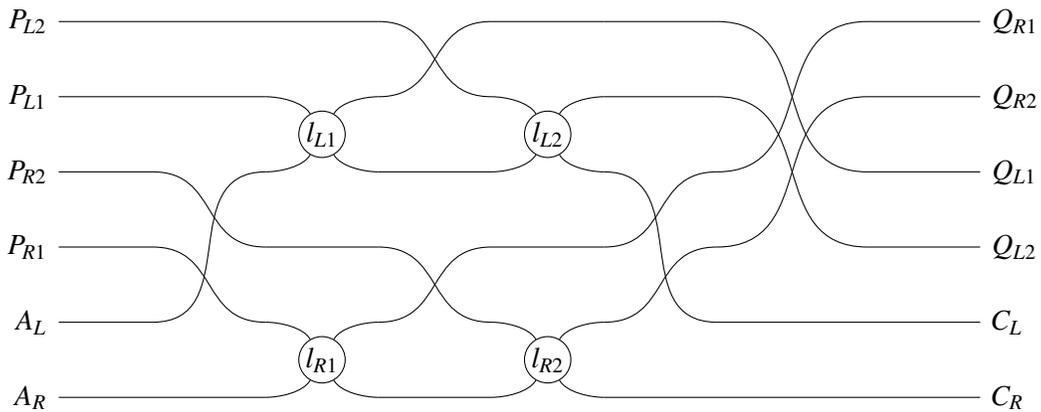
\begin{center}
    \begin{tikzpicture}[scale=2]
	\begin{pgfonlayer}{nodelayer}
		\node [style=none, label={left:$P_{L2}$}] (9) at (-4.25, 3) {};
		\node [style=none, label={left:$P_{L1}$}] (10) at (-4.25, 2) {};
		\node [style=none, label={left:$A_L$}] (12) at (-4.25, -1) {};
		\node [style=none] (13) at (0, 3) {};
		\node [style=none] (15) at (0, 1) {};
		\node [style=none] (16) at (1.5, 3) {};
		\node [style=none] (18) at (1.5, 1) {};
		\node [style=none, label={right:$Q_{L1}$}] (21) at (8, 1) {};
		\node [style=none, label={right:$Q_{L2}$}] (22) at (8, 0) {};
		\node [style=none, label={right:$C_L$}] (23) at (8, -1) {};
		\node [style=none] (25) at (0, 2) {};
		\node [style=none] (26) at (1.5, 2) {};
		\node [style=none, label={left:$P_{R2}$}] (27) at (-4.25, 1) {};
		\node [style=none, label={left:$P_{R1}$}] (28) at (-4.25, 0) {};
		\node [style=none, label={left:$A_R$}] (29) at (-4.25, -2) {};
		\node [style=none] (30) at (0, 0) {};
		\node [style=none] (31) at (0, -2) {};
		\node [style=none] (32) at (1.5, 0) {};
		\node [style=none] (33) at (1.5, -2) {};
		\node [style=none, label={right:$Q_{R1}$}] (35) at (8, 3) {};
		\node [style=none, label={right:$Q_{R2}$}] (36) at (8, 2) {};
		\node [style=none, label={right:$C_R$}] (37) at (8, -2) {};
		\node [style=none] (39) at (0, -1) {};
		\node [style=none] (40) at (1.5, -1) {};
		\node [style=none] (41) at (-1.5, 2) {};
		\node [style=morphism] (42) at (-0.75, 1.5) {$l_{L1}$};
		\node [style=none] (44) at (-1.5, -2) {};
		\node [style=morphism] (45) at (-0.75, -1.5) {$l_{R1}$};
		\node [style=none] (47) at (-3, 1) {};
		\node [style=none] (48) at (-1.5, 1) {};
		\node [style=none] (49) at (-3, -1) {};
		\node [style=none] (50) at (-1.5, 0) {};
		\node [style=none] (51) at (-3, 0) {};
		\node [style=none] (52) at (-1.5, -1) {};
		\node [style=morphism] (54) at (2.25, 1.5) {$l_{L2}$};
		\node [style=none] (55) at (3, 2) {};
		\node [style=none] (56) at (3, -2) {};
		\node [style=morphism] (57) at (2.25, -1.5) {$l_{R2}$};
		\node [style=none] (60) at (3, 3) {};
		\node [style=none] (65) at (3, -1) {};
		\node [style=none] (66) at (3, 1) {};
		\node [style=none] (67) at (4.5, -1) {};
		\node [style=none] (68) at (3, 0) {};
		\node [style=none] (72) at (4.5, 1) {};
		\node [style=none] (73) at (6.5, 3) {};
		\node [style=none] (74) at (4.5, 0) {};
		\node [style=none] (75) at (6.5, 2) {};
		\node [style=none] (76) at (4.5, 3) {};
		\node [style=none] (77) at (6.5, 1) {};
		\node [style=none] (78) at (4.5, 2) {};
		\node [style=none] (79) at (6.5, 0) {};
	\end{pgfonlayer}
	\begin{pgfonlayer}{edgelayer}
		\draw (9.center) to (13.center);
		\draw (15.center) to (18.center);
		\draw [in=0, out=180, looseness=1.25] (26.center) to (13.center);
		\draw [in=-180, out=0, looseness=1.25] (25.center) to (16.center);
		\draw (31.center) to (33.center);
		\draw [in=0, out=180, looseness=1.25] (40.center) to (30.center);
		\draw [in=-180, out=0, looseness=1.25] (39.center) to (32.center);
		\draw [in=0, out=120, looseness=0.75] (42) to (41.center);
		\draw [in=-180, out=-60, looseness=0.75] (42) to (15.center);
		\draw [in=-180, out=60, looseness=0.75] (42) to (25.center);
		\draw [in=0, out=-120, looseness=0.75] (45) to (44.center);
		\draw [in=-180, out=-60, looseness=0.75] (45) to (31.center);
		\draw [in=-180, out=60, looseness=0.75] (45) to (39.center);
		\draw [in=0, out=180, looseness=1.25] (50.center) to (47.center);
		\draw [in=-180, out=0, looseness=1.25] (49.center) to (48.center);
		\draw (10.center) to (41.center);
		\draw [in=0, out=180, looseness=1.25] (52.center) to (51.center);
		\draw (27.center) to (47.center);
		\draw (28.center) to (51.center);
		\draw (12.center) to (49.center);
		\draw (29.center) to (44.center);
		\draw (50.center) to (30.center);
		\draw [in=180, out=60, looseness=0.75] (54) to (55.center);
		\draw [in=240, out=0, looseness=0.75] (18.center) to (54);
		\draw [in=0, out=120, looseness=0.75] (54) to (26.center);
		\draw [in=-60, out=180, looseness=0.75] (56.center) to (57);
		\draw [in=240, out=0, looseness=0.75] (33.center) to (57);
		\draw [in=0, out=120, looseness=0.75] (57) to (40.center);
		\draw (56.center) to (37.center);
		\draw (16.center) to (60.center);
		\draw [in=0, out=180, looseness=1.25] (67.center) to (66.center);
		\draw [in=-60, out=180, looseness=0.75] (66.center) to (54);
		\draw [in=180, out=60, looseness=0.75] (57) to (65.center);
		\draw (32.center) to (68.center);
		\draw (67.center) to (23.center);
		\draw [in=0, out=180, looseness=1.25] (73.center) to (72.center);
		\draw [in=0, out=180, looseness=1.25] (75.center) to (74.center);
		\draw [in=0, out=180, looseness=1.25] (77.center) to (76.center);
		\draw [in=0, out=180, looseness=1.25] (79.center) to (78.center);
		\draw (73.center) to (35.center);
		\draw (75.center) to (36.center);
		\draw (77.center) to (21.center);
		\draw (79.center) to (22.center);
		\draw [in=0, out=-120, looseness=0.75] (42) to (48.center);
		\draw [in=0, out=120, looseness=0.75] (45) to (52.center);
		\draw [in=-180, out=0, looseness=1.25] (65.center) to (74.center);
		\draw [in=0, out=180, looseness=1.25] (72.center) to (68.center);
		\draw (60.center) to (76.center);
		\draw (55.center) to (78.center);
	\end{pgfonlayer}
\end{tikzpicture}
  \end{center}
  
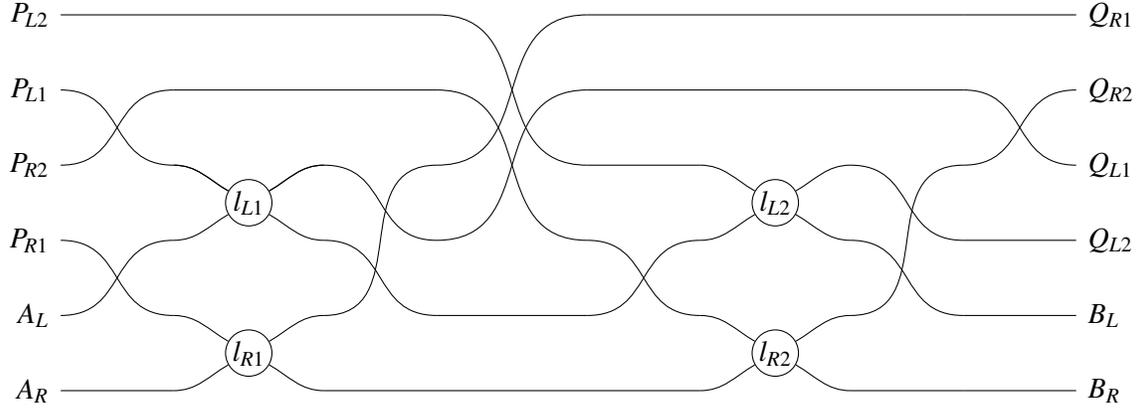
\vspace{10pt}
  \begin{center}
    \begin{tikzpicture}[scale=2]
	\begin{pgfonlayer}{nodelayer}
		\node [style=none, label={left:$P_{R1}$}] (9) at (-7.5, -0.75) {};
		\node [style=none, label={left:$P_{L1}$}] (10) at (-7.5, 1.25) {};
		\node [style=none, label={left:$A_R$}] (12) at (-7.5, -2.75) {};
		\node [style=none, label={left:$A_L$}] (27) at (-7.5, -1.75) {};
		\node [style=none] (31) at (-6, -2.75) {};
		\node [style=none] (35) at (-4, -2.75) {};
		\node [style=none] (44) at (4.5, 2.25) {};
		\node [style=none, label={left:$P_{R2}$}] (45) at (-7.5, 0.25) {};
		\node [style=none, label={left:$P_{L2}$}] (46) at (-7.5, 2.25) {};
		\node [style=none] (51) at (1, 0.25) {};
		\node [style=none] (52) at (1, -2.75) {};
		\node [style=none] (58) at (4.5, -1.75) {};
		\node [style=none] (59) at (4.5, -2.75) {};
		\node [style=none] (60) at (4.5, -0.75) {};
		\node [style=none] (69) at (-6, -0.75) {};
		\node [style=morphism] (72) at (-5, -2.25) {$l_{R1}$};
		\node [style=none] (73) at (-6, -1.75) {};
		\node [style=morphism] (77) at (-5, -0.25) {$l_{L1}$};
		\node [style=none] (92) at (3, -2.75) {};
		\node [style=morphism] (93) at (2, -0.25) {$l_{L2}$};
		\node [style=morphism] (96) at (2, -2.25) {$l_{R2}$};
		\node [style=none] (101) at (3, -0.75) {};
		\node [style=none] (102) at (3, -1.75) {};
		\node [style=none] (104) at (3, 0.25) {};
		\node [style=none] (111) at (-0.5, 2.25) {};
		\node [style=none] (113) at (-0.5, 1.25) {};
		\node [style=none] (114) at (-2.5, 2.25) {};
		\node [style=none] (115) at (-0.5, 0.25) {};
		\node [style=none] (116) at (-2.5, 1.25) {};
		\node [style=none] (119) at (-2.5, -1.75) {};
		\node [style=none] (120) at (-2.5, 0.25) {};
		\node [style=none] (121) at (-4, -1.75) {};
		\node [style=none] (122) at (-2.5, -0.75) {};
		\node [style=none] (124) at (-4, 0.25) {};
		\node [style=none] (125) at (-4, -0.75) {};
		\node [style=none] (126) at (-0.5, -0.75) {};
		\node [style=none] (127) at (1, -0.75) {};
		\node [style=none] (128) at (-0.5, -1.75) {};
		\node [style=none] (129) at (1, -1.75) {};
		\node [style=none] (131) at (-6, 1.25) {};
		\node [style=none] (133) at (-6, 0.25) {};
		\node [style=none, label={right:$Q_{L1}$}] (134) at (6, 0.25) {};
		\node [style=none, label={right:$Q_{R1}$}] (135) at (6, 2.25) {};
		\node [style=none, label={right:$B_L$}] (136) at (6, -1.75) {};
		\node [style=none, label={right:$B_R$}] (137) at (6, -2.75) {};
		\node [style=none, label={right:$Q_{L2}$}] (138) at (6, -0.75) {};
		\node [style=none, label={right:$Q_{R2}$}] (139) at (6, 1.25) {};
		\node [style=none] (140) at (4.5, 1.25) {};
		\node [style=none] (142) at (4.5, 0.25) {};
	\end{pgfonlayer}
	\begin{pgfonlayer}{edgelayer}
		\draw (12.center) to (31.center);
		\draw [in=0, out=150] (72) to (73.center);
		\draw [in=180, out=-30] (72) to (35.center);
		\draw [in=0, out=-150] (72) to (31.center);
		\draw [in=-150, out=0] (69.center) to (77);
		\draw (35.center) to (52.center);
		\draw [in=180, out=-30] (96) to (92.center);
		\draw [in=150, out=0] (51.center) to (93);
		\draw [in=0, out=-150] (96) to (52.center);
		\draw [in=180, out=30] (96) to (102.center);
		\draw [in=30, out=180] (104.center) to (93);
		\draw [in=-30, out=180] (101.center) to (93);
		\draw [in=0, out=180, looseness=1.25] (58.center) to (101.center);
		\draw [in=0, out=180, looseness=1.25] (60.center) to (104.center);
		\draw (92.center) to (59.center);
		\draw [in=0, out=180, looseness=1.25] (115.center) to (114.center);
		\draw [in=180, out=0, looseness=1.25] (121.center) to (120.center);
		\draw [in=-180, out=30] (72) to (121.center);
		\draw [in=30, out=180] (124.center) to (77);
		\draw [in=-30, out=180] (125.center) to (77);
		\draw [in=-180, out=30] (77) to (124.center);
		\draw [in=0, out=180, looseness=1.25] (119.center) to (125.center);
		\draw [in=180, out=0, looseness=1.25] (124.center) to (122.center);
		\draw [in=0, out=180, looseness=1.25] (111.center) to (120.center);
		\draw [in=0, out=180, looseness=1.25] (113.center) to (122.center);
		\draw [in=0, out=180, looseness=1.25] (129.center) to (126.center);
		\draw [in=-180, out=0, looseness=1.25] (128.center) to (127.center);
		\draw [in=-150, out=0] (127.center) to (93);
		\draw [in=0, out=150] (96) to (129.center);
		\draw [in=0, out=180, looseness=1.25] (126.center) to (116.center);
		\draw (119.center) to (128.center);
		\draw (115.center) to (51.center);
		\draw (111.center) to (44.center);
		\draw (46.center) to (114.center);
		\draw [in=0, out=180, looseness=1.25] (73.center) to (9.center);
		\draw [in=-180, out=0, looseness=1.25] (27.center) to (69.center);
		\draw [in=0, out=150] (77) to (133.center);
		\draw [in=150, out=0] (133.center) to (77);
		\draw [in=0, out=180, looseness=1.25] (133.center) to (10.center);
		\draw [in=-180, out=0, looseness=1.25] (45.center) to (131.center);
		\draw (131.center) to (116.center);
		\draw [in=-180, out=0, looseness=1.25] (102.center) to (142.center);
		\draw (113.center) to (140.center);
		\draw [in=0, out=180, looseness=1.25] (134.center) to (140.center);
		\draw [in=-180, out=0, looseness=1.25] (142.center) to (139.center);
		\draw (44.center) to (135.center);
		\draw (60.center) to (138.center);
		\draw (136.center) to (58.center);
		\draw (59.center) to (137.center);
	\end{pgfonlayer}
\end{tikzpicture}
  \end{center}
  The latter diagram is almost exactly the correct one, other than the
  outer twists on the pairs $(P_{L1}, P_{R2})$ and $(Q_{R2}, Q_{L1})$.
  These are cancelled with similar twists in the right component using
  the coend relations.

  As in the $\Optic$ setting, the structure morphisms (including the
  symmetry isomorphism) are constructed as the image of the structure
  morphisms of $\C \times \C^\op$, and functoriality of the inclusion
  guarantees that the equations of a symmetric monoidal category hold.
\end{proof}

\section{Duality}

For extensional learners, there is an obvious involution
$(-)^* : \Learn \to \Learn^{\op}$ that comes into view by virtue of
the new symmetrical formulation: just switch the two components!

\begin{proposition}
  There is a (strict) involutive symmetric monoidal functor
  $(-)^* : \Learn \to \Learn^{\op}$ given by $(A, A')^* := (A', A)$
  and $(L \mid R)^* := (R \mid L)$.
\end{proposition}
\begin{proof}
  On objects, we have
  \begin{align*}
    \left( (A, A') \otimes (B, B') \right)^*
    = (A \otimes B, B' \otimes A')^*
    = (B' \otimes A', A \otimes B)
    = (B', B) \otimes (A', A)
    = (B, B')^* \otimes (A, A')^*
  \end{align*}
  Preservation of identity, composition and tensor of morphisms is
  clear by observing that their definitions are exactly symmetric
  between each component.
\end{proof}
The resulting strictness of this functor is why it is convenient to
switch the objects in the right component of
$(A, A') \otimes (B, B') :\defeq (A \otimes B, B' \otimes A')$.

\begin{remark}
Let us calculate how this duality acts on intensional learners.
Stepping through the equivalence between intensional learners and
their coend formulation, the dual of a learner
$(P, I, U, r) : (A, A') \hto (B, B')$ is a learner
$(P^*, I^*, U^*, r^*) : (B', B) \hto (A', A)$ with parameter set
$P^* := P \times A$, and
\begin{alignat*}{3}
&I^* : (P \times A) \times B' &&\to A' \\
&I^*((p, p_a), b') &&:= r(p, p_a, b') \\
&U^* : (P \times A) \times B' \times A &&\to (P \times A) \\
&U^*((p, p_a), b', a) &&:= (U(p, p_a, b'), a) \\
&r^* : (P \times A) \times B' \times A &&\to B \\
&r^*((p, p_a), b', a) &&:= I(U(p, p_a, b'), a)
\end{alignat*}
If we do this twice, we return to a learner $(A, A') \hto (B, B')$
with parameter set $P^{**} := P \times A \times B'$, and
\begin{alignat*}{3}
&I^{**} : (P \times A \times B') \times A &&\to B \\
&I^{**}((p, p_a, p_{b'}), a) &&:= I(U(p, p_a, p_{b'}), a) \\
&U^{**} : (P \times A \times B') \times A \times B' &&\to (P \times A \times B') \\
&U^{**}((p, p_a, p_{b'}), a, b') &&:= (U(p, p_a, p_{b'}), a, b') \\
&r^{**} : (P \times A \times B') \times A \times B' &&\to A' \\
&r^{**}((p, p_a, p_{b'}), a, b') &&:= r(U(p, p_a, p_{b'}), a, b')
\end{alignat*}
Under intensional equivalence this is not equal to the original
learner, but under extensional equivalence it is: use
$U : P \times A \times B' \to P$ as the function $f$ between parameter
sets, and the identity
$\id : P \times A \times B' \to P \times A \times B'$ as the diagonal
filler $\hat{U}$. This double-dual learner is extensionally identical
to the original learner, but intensionally it is `running one training
datum behind': when fed a new element of $A \times B'$, it updates $P$
using the pair remembered in the parameter set and stores the provided
element of $A \times B'$ for next time.
\end{remark}

An object $(A, A')$ and its counterpart $(A', A)$ are related by cup
and cap morphisms:

\begin{definition}
  For any object $(A, A')$, define the cup
  \begin{align*}
    \eta_{(A, A')} : (I, I) \hto (A, A') \otimes (A, A')^* = (A \otimes A', A \otimes A')
  \end{align*}
  as the learner with $P := A \otimes A'$ and $Q := I$ and the obvious
  maps
  \begin{align*}
    (A \otimes A') \otimes I &\to I \otimes (A \otimes A') \\
    I \otimes (A \otimes A') &\to (A \otimes A') \otimes I
  \end{align*}
  and define the cap
  \begin{align*}
    \varepsilon_{(A, A')} : (A' \otimes A, A' \otimes A) = (A, A')^* \otimes (A, A') \hto (I, I)
  \end{align*}
  as the learner with $P := I$ and $Q := A' \otimes A$ and the obvious maps
  \begin{align*}
    I \otimes (A' \otimes A) &\to (A' \otimes A) \otimes I \\
    (A' \otimes A) \otimes I &\to I \otimes (A' \otimes A).
  \end{align*}
\end{definition}
These maps are dual, in that:
\begin{align*}
  (\eta_{(A, A')})^* = \varepsilon_{(A, A')^*} \qquad (\varepsilon_{(A, A')})^* = \eta_{(A, A')^*}
\end{align*}

All signs are pointing to $(A, A')^*$ being the actual monoidal dual
of $(A, A')$. Tragically,
\begin{proposition}
  $\Learn_\C$ is typically not compact closed.
\end{proposition}
\begin{proof}
  Consider an object of the form $(A, I)$. After cancelling some
  monoidal units, the composite learner
  \begin{align*}
    (A, I) \hto (A, I) \otimes ((A, I)^* \otimes (A, I)) \hto ((A, I)
    \otimes (A, I)^*) \otimes (A, I) \hto (A, I)
  \end{align*}
  corresponding to the snake equation for $(A, I)$ has $P := A$ and
  $Q := A$, and components
  \begin{align*}
    A \otimes A &\xrightarrow{s_{A, A}} A \otimes A \\
    A \otimes I &\xrightarrow{\id} A \otimes I.
  \end{align*}
  This is typically not equal to the identity learner; a priori there
  are no non-trivial maps involving $A$ whatsoever, as can be seen by
  setting $\C$ to be a commutative monoid considered as a discrete
  category.

  As an intensional learner, the above learner is similar to the
  identity map but with its output ``delayed one step'', and has
  $P := A$ with maps $I(p, a) := p$ and  $U(p, b, a) := a$.
\end{proof}

Learners do have the following weaker structure, much like the
teleological categories discussed in \cite[Section
5]{CoherenceForLenses} and \cite[Section 2.1]{riley:optics} but with
cups as well as caps.

\begin{proposition}
  The families of morphisms
  \begin{align*}
    \eta_{(A, A')} &: (I, I) \hto (A, A') \otimes (A, A')^* \\
    \varepsilon_{(A, A')} &: (A, A')^* \otimes (A, A') \hto (I, I)
  \end{align*}
  are monoidal, and are extranatural with respect to morphisms in the
  image of $\iota : \C \times \C^\op \to \Learn_\C$. Furthermore,
  $s_{(A, A'),(A, A')^*}\eta_{(A, A')} = \eta_{(A, A')^*}$ and
  $\varepsilon_{(A, A')} s_{(A, A'),(A, A')^*} = \varepsilon_{(A,
    A')^*}$, where $s_{(A, A'),(A, A')^*}$ denotes the symmetry
  morphism.
\end{proposition}
\begin{proof}
  The definitions of $\eta$ and $\epsilon$ are so simple that these
  are all easily checked by direct calculation.
\end{proof}

We have a canonical decomposition of any learner using these
morphisms, in the same style as \cite[Proposition
2.1.10]{riley:optics}.
\begin{lemma}\label{lem:learner-decomp}
  Every learner $(l \mid r) : (S, S') \hto (A, A')$ is equal to the
  learner described by the diagram
  
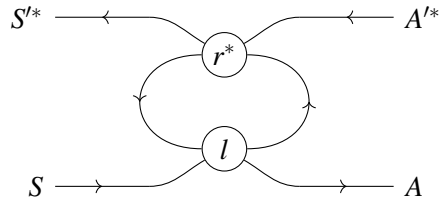
\begin{center}
    \begin{tikzpicture}[scale=2]
	\begin{pgfonlayer}{nodelayer}
		\node [style=none, label={left:$S$}] (10) at (-2.25, -1.25) {};
		\node [style=none, label={left:$S'^*$}] (12) at (-2.25, 1) {};
		\node [style=morphism] (20) at (0, -0.75) {$l$};
		\node [style=morphism] (28) at (0, 0.5) {$r^*$};
		\node [style=none] (30) at (-1, -1.25) {};
		\node [style=none] (31) at (-1, 1) {};
		\node [style=none] (34) at (1, -1.25) {};
		\node [style=none] (35) at (1, 1) {};
		\node [style=none, label={right:$A'^*$}] (42) at (2.25, 1) {};
		\node [style=none, label={right:$A$}] (44) at (2.25, -1.25) {};
	\end{pgfonlayer}
	\begin{pgfonlayer}{edgelayer}
		\draw [style=midarrow] (10.center) to (30.center);
		\draw [style=midarrow] (31.center) to (12.center);
		\draw [in=-180, out=30] (28) to (35.center);
		\draw [in=-150, out=0] (30.center) to (20);
		\draw [in=-180, out=-30] (20) to (34.center);
		\draw [in=0, out=150] (28) to (31.center);
		\draw [style=midarrow] (42.center) to (35.center);
		\draw [style=midarrow] (34.center) to (44.center);
		\draw [style=midarrow, in=360, out=0, looseness=2.25] (20) to (28);
		\draw [style=midarrow, in=-180, out=180, looseness=2.25] (28) to (20);
	\end{pgfonlayer}
\end{tikzpicture}
  \end{center}
  recalling that we silently include objects and morphisms of $\C$
  in $\Learn_\C$ using $\iota$.
\end{lemma}
\begin{proof}
  This can be checked by direct calculation, cancelling all the
  occurrences of the monoidal unit created by the use of $\iota$.
\end{proof}

The snake equations are precisely what is missing for $\Learn_\C$ to
be compact closed.
\begin{definition}
  The category $\Atemp_\C$ of \emph{atemporal learners} is the
  quotient of $\Learn_\C$ obtained by identifying the snake diagram
  \begin{align*}
    (A, A') \hto ((A, A') \otimes (A, A')^*) \otimes (A, A') \hto (A,
    A') \otimes ((A, A')^* \otimes (A, A')) \hto (A, A')
  \end{align*}
  with the identity morphism for every object $(A, A')$.
\end{definition}

Trivialising just one of the snake diagrams also trivialises the
other: the dual snake diagram is the image of this one under the
strictly monoidal functor $-^*$.

\begin{proposition}\label{prop:atemp-closed}
  $\Atemp_\C$ is compact closed.
\end{proposition}
\begin{proof}
  The only missing property is extranaturality of the cup and cap with
  respect to \emph{all} morphisms in $\Atemp_\C$. The snake equations
  holding implies that $\eta$ and $\varepsilon$ are extranatural with
  respect to each other:
  
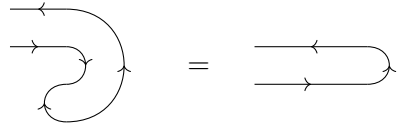
\begin{center}
    \begin{tikzpicture}[scale=2]
	\begin{pgfonlayer}{nodelayer}
		\node [style=none] (34) at (0, -0.75) {};
		\node [style=none] (35) at (0, -0.25) {};
		\node [style=none] (36) at (0, 0.25) {};
		\node [style=none] (37) at (0, 0.75) {};
		\node [style=none] (38) at (-0.75, 0.25) {};
		\node [style=none] (39) at (-0.75, 0.75) {};
	\end{pgfonlayer}
	\begin{pgfonlayer}{edgelayer}
		\draw [style=midarrow, in=0, out=0, looseness=1.75] (34.center) to (37.center);
		\draw [style=midarrow] (38.center) to (36.center);
		\draw [style=midarrow] (37.center) to (39.center);
		\draw [style=midarrow, in=-360, out=0, looseness=1.75] (36.center) to (35.center);
		\draw [style=midarrow, in=-180, out=180, looseness=2.00] (34.center) to (35.center);
	\end{pgfonlayer}
\end{tikzpicture}
    \quad $=$ \quad
    \begin{tikzpicture}[scale=2]
	\begin{pgfonlayer}{nodelayer}
		\node [style=none] (36) at (0.75, -0.25) {};
		\node [style=none] (37) at (0.75, 0.25) {};
		\node [style=none] (38) at (-0.75, -0.25) {};
		\node [style=none] (39) at (-0.75, 0.25) {};
	\end{pgfonlayer}
	\begin{pgfonlayer}{edgelayer}
		\draw [style=midarrow] (38.center) to (36.center);
		\draw [style=midarrow] (37.center) to (39.center);
		\draw [style=midarrow, bend right=90, looseness=2.00] (36.center) to (37.center);
	\end{pgfonlayer}
\end{tikzpicture}
  \end{center}
  and similarly with the roles reversed.

  From Lemma~\ref{lem:learner-decomp}, it follows that $\eta$ and
  $\varepsilon$ are natural with respect to \emph{every} map in
  $\Atemp_\C$. Decomposing a map as via the Lemma, we can then see:
  
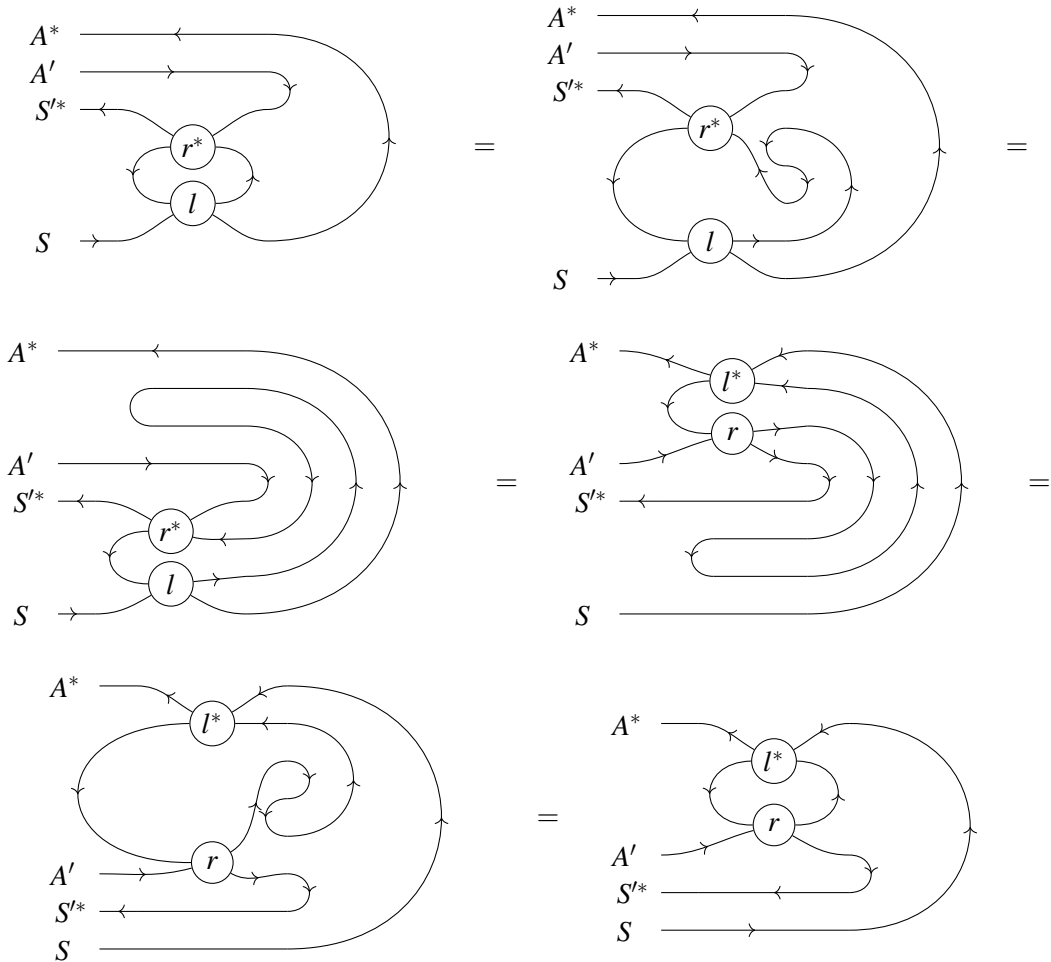
\begin{center}
    \begin{tikzpicture}[scale=2]
	\begin{pgfonlayer}{nodelayer}
		\node [style=none, label={left:$S\phantom{'^*}$}] (10) at (-1.5, -1.25) {};
		\node [style=none, label={left:$S'^*$}] (12) at (-1.5, 0.5) {};
		\node [style=morphism] (20) at (0, -0.75) {$l$};
		\node [style=morphism] (28) at (0, 0) {$r^*$};
		\node [style=none] (30) at (-1, -1.25) {};
		\node [style=none] (31) at (-1, 0.5) {};
		\node [style=none] (34) at (1, -1.25) {};
		\node [style=none] (35) at (1, 0.5) {};
		\node [style=none] (36) at (1, 1) {};
		\node [style=none] (37) at (1, 1.5) {};
		\node [style=none, label={left:$A'\phantom{^*}$}] (38) at (-1.5, 1) {};
		\node [style=none, label={left:$A^*\phantom{'}$}] (39) at (-1.5, 1.5) {};
	\end{pgfonlayer}
	\begin{pgfonlayer}{edgelayer}
		\draw [style=midarrow] (10.center) to (30.center);
		\draw [style=midarrow] (31.center) to (12.center);
		\draw [in=-180, out=30] (28) to (35.center);
		\draw [in=-150, out=0] (30.center) to (20);
		\draw [in=180, out=-30] (20) to (34.center);
		\draw [in=0, out=150] (28) to (31.center);
		\draw [style=midarrow, bend left=270, looseness=2.25] (28) to (20);
		\draw [style=midarrow, bend right=90, looseness=2.00] (34.center) to (37.center);
		\draw [style=midarrow] (38.center) to (36.center);
		\draw [style=midarrow] (37.center) to (39.center);
		\draw [style=midarrow, bend right=90, looseness=2.25] (20) to (28);
		\draw [style=midarrow, bend left=90, looseness=2.00] (36.center) to (35.center);
	\end{pgfonlayer}
\end{tikzpicture}
    \quad $=$ \quad
    \begin{tikzpicture}[scale=2]
	\begin{pgfonlayer}{nodelayer}
		\node [style=none, label={left:$S\phantom{'^*}$}] (10) at (-1.5, -1.75) {};
		\node [style=none, label={left:$S'^*$}] (12) at (-1.5, 0.75) {};
		\node [style=morphism] (20) at (0, -1.25) {$l$};
		\node [style=morphism] (28) at (0, 0.25) {$r^*$};
		\node [style=none] (30) at (-1, -1.75) {};
		\node [style=none] (31) at (-1, 0.75) {};
		\node [style=none] (34) at (1, -1.75) {};
		\node [style=none] (35) at (1, 0.75) {};
		\node [style=none] (36) at (1, 1.25) {};
		\node [style=none] (37) at (1, 1.75) {};
		\node [style=none, label={left:$A'\phantom{^*}$}] (38) at (-1.5, 1.25) {};
		\node [style=none, label={left:$A^*\phantom{'}$}] (39) at (-1.5, 1.75) {};
		\node [style=none] (40) at (1, -1.25) {};
		\node [style=none] (41) at (1, 0.25) {};
		\node [style=none] (42) at (1, -0.25) {};
		\node [style=none] (43) at (1, -0.75) {};
	\end{pgfonlayer}
	\begin{pgfonlayer}{edgelayer}
		\draw [style=midarrow] (10.center) to (30.center);
		\draw [style=midarrow] (31.center) to (12.center);
		\draw [in=-180, out=30] (28) to (35.center);
		\draw [in=-150, out=0] (30.center) to (20);
		\draw [in=180, out=-30] (20) to (34.center);
		\draw [in=0, out=150] (28) to (31.center);
		\draw [style=midarrow, bend left=270, looseness=2.25] (28) to (20);
		\draw [style=midarrow, bend right=90, looseness=2.00] (34.center) to (37.center);
		\draw [style=midarrow] (38.center) to (36.center);
		\draw [style=midarrow] (37.center) to (39.center);
		\draw [style=midarrow] (20) to (40.center);
		\draw [style=midarrow, bend right=90, looseness=2.00] (40.center) to (41.center);
		\draw [style=midarrow, bend right=90, looseness=1.75] (41.center) to (42.center);
		\draw [style=midarrow, bend left=90, looseness=2.00] (42.center) to (43.center);
		\draw [style=midarrow, in=-15, out=165, looseness=0.75] (43.center) to (28);
		\draw [style=midarrow, bend left=90, looseness=2.00] (36.center) to (35.center);
	\end{pgfonlayer}
\end{tikzpicture} $=$
  \end{center}
  \begin{center}
    \begin{tikzpicture}[scale=2]
	\begin{pgfonlayer}{nodelayer}
		\node [style=none, label={left:$S\phantom{'^*}$}] (10) at (-1.5, -1.75) {};
		\node [style=none, label={left:$S'^*$}] (12) at (-1.5, -0.25) {};
		\node [style=morphism] (20) at (0, -1.35) {$l$};
		\node [style=morphism] (28) at (0, -0.65) {$r^*$};
		\node [style=none] (30) at (-1, -1.75) {};
		\node [style=none] (31) at (-1, -0.25) {};
		\node [style=none] (34) at (1, -1.75) {};
		\node [style=none] (35) at (1, -0.25) {};
		\node [style=none] (36) at (1, 0.25) {};
		\node [style=none] (37) at (1, 1.75) {};
		\node [style=none, label={left:$A'\phantom{^*}$}] (38) at (-1.5, 0.25) {};
		\node [style=none, label={left:$A^*\phantom{'}$}] (39) at (-1.5, 1.75) {};
		\node [style=none] (40) at (1, -1.25) {};
		\node [style=none] (41) at (1, 1.25) {};
		\node [style=none] (42) at (1, 0.75) {};
		\node [style=none] (43) at (1, -0.75) {};
		\node [style=none] (44) at (-0.25, 0.75) {};
		\node [style=none] (45) at (-0.25, 1.25) {};
	\end{pgfonlayer}
	\begin{pgfonlayer}{edgelayer}
		\draw [style=midarrow] (10.center) to (30.center);
		\draw [style=midarrow] (31.center) to (12.center);
		\draw [in=-180, out=30] (28) to (35.center);
		\draw [in=-150, out=0] (30.center) to (20);
		\draw [in=180, out=-30] (20) to (34.center);
		\draw [in=0, out=150] (28) to (31.center);
		\draw [style=midarrow, bend right=90, looseness=2.50] (28) to (20);
		\draw [style=midarrow, bend right=90, looseness=2.00] (34.center) to (37.center);
		\draw [style=midarrow] (38.center) to (36.center);
		\draw [style=midarrow] (37.center) to (39.center);
		\draw [style=midarrow] (20) to (40.center);
		\draw [style=midarrow, bend right=90, looseness=2.00] (40.center) to (41.center);
		\draw [style=midarrow, bend left=90, looseness=2.00] (42.center) to (43.center);
		\draw [style=midarrow, in=-15, out=180, looseness=0.75] (43.center) to (28);
		\draw [style=midarrow, bend left=90, looseness=2.00] (36.center) to (35.center);
		\draw (44.center) to (42.center);
		\draw (41.center) to (45.center);
		\draw [bend left=270, looseness=2.00] (45.center) to (44.center);
	\end{pgfonlayer}
\end{tikzpicture}
    \quad $=$ \quad
    \begin{tikzpicture}[scale=2]
	\begin{pgfonlayer}{nodelayer}
		\node [style=none, label={left:$S\phantom{'^*}$}] (10) at (-1.5, -1.75) {};
		\node [style=none, label={left:$S'^*$}] (12) at (-1.5, -0.25) {};
		\node [style=morphism] (20) at (0, 1.35) {$l^*$};
		\node [style=morphism] (28) at (0, 0.65) {$r$};
		\node [style=none] (31) at (-1, -0.25) {};
		\node [style=none] (34) at (1, -1.75) {};
		\node [style=none] (35) at (1, -0.25) {};
		\node [style=none] (36) at (1, 0.25) {};
		\node [style=none] (37) at (1, 1.75) {};
		\node [style=none, label={left:$A'\phantom{^*}$}] (38) at (-1.5, 0.25) {};
		\node [style=none, label={left:$A^*\phantom{'}$}] (39) at (-1.5, 1.75) {};
		\node [style=none] (40) at (1, -1.25) {};
		\node [style=none] (41) at (1, 1.25) {};
		\node [style=none] (42) at (1, 0.75) {};
		\node [style=none] (43) at (1, -0.75) {};
		\node [style=none] (46) at (-0.25, -1.25) {};
		\node [style=none] (47) at (-0.25, -0.75) {};
	\end{pgfonlayer}
	\begin{pgfonlayer}{edgelayer}
		\draw [style=midarrow] (31.center) to (12.center);
		\draw [style=midarrow, bend right=90, looseness=2.00] (34.center) to (37.center);
		\draw [style=midarrow, bend right=90, looseness=2.00] (40.center) to (41.center);
		\draw [style=midarrow, bend left=90, looseness=2.00] (42.center) to (43.center);
		\draw [style=midarrow, bend left=90, looseness=2.00] (36.center) to (35.center);
		\draw (10.center) to (34.center);
		\draw [style=midarrow, bend left=270, looseness=2.00] (47.center) to (46.center);
		\draw (46.center) to (40.center);
		\draw (43.center) to (47.center);
		\draw (35.center) to (31.center);
		\draw [style=midarrow, bend right=90, looseness=2.75] (20) to (28);
		\draw [style=midarrow, in=-330, out=-180] (37.center) to (20);
		\draw [style=midarrow, in=0, out=165] (20) to (39.center);
		\draw [style=midarrow] (41.center) to (20);
		\draw [style=midarrow] (28) to (42.center);
		\draw [style=midarrow, in=-180, out=-30] (28) to (36.center);
		\draw [style=midarrow, in=-165, out=0] (38.center) to (28);
	\end{pgfonlayer}
\end{tikzpicture} $=$
  \end{center}
  \begin{center}
    \begin{tikzpicture}[scale=2]
	\begin{pgfonlayer}{nodelayer}
		\node [style=none, label={left:$S\phantom{'^*}$}] (10) at (-1.5, -1.75) {};
		\node [style=none, label={left:$S'^*$}] (12) at (-1.5, -1.25) {};
		\node [style=morphism] (20) at (0, 1.25) {$l^*$};
		\node [style=morphism] (28) at (0, -0.6) {$r$};
		\node [style=none] (31) at (-1, -1.25) {};
		\node [style=none] (34) at (1, -1.75) {};
		\node [style=none] (35) at (1, -1.25) {};
		\node [style=none] (36) at (1, -0.75) {};
		\node [style=none] (37) at (1, 1.75) {};
		\node [style=none, label={left:$A'\phantom{^*}$}] (38) at (-1.5, -0.75) {};
		\node [style=none, label={left:$A^*\phantom{'}$}] (39) at (-1.5, 1.75) {};
		\node [style=none] (40) at (1, -0.25) {};
		\node [style=none] (41) at (1, 1.25) {};
		\node [style=none] (42) at (1, 0.75) {};
		\node [style=none] (43) at (1, 0.25) {};
		\node [style=none] (44) at (-1, 1.75) {};
	\end{pgfonlayer}
	\begin{pgfonlayer}{edgelayer}
		\draw [style=midarrow] (31.center) to (12.center);
		\draw [style=midarrow, bend right=90, looseness=2.00] (34.center) to (37.center);
		\draw [style=midarrow, bend right=90, looseness=2.00] (40.center) to (41.center);
		\draw [style=midarrow, bend left=90, looseness=2.00] (42.center) to (43.center);
		\draw [style=midarrow, bend left=90, looseness=2.00] (36.center) to (35.center);
		\draw (10.center) to (34.center);
		\draw (35.center) to (31.center);
		\draw [style=midarrow, bend right=90, looseness=2.75] (20) to (28);
		\draw [style=midarrow, in=-330, out=-180] (37.center) to (20);
		\draw [style=midarrow] (41.center) to (20);
		\draw [style=midarrow, in=-180, out=30] (28) to (42.center);
		\draw [style=midarrow, in=-180, out=-30] (28) to (36.center);
		\draw [style=midarrow, in=-165, out=0] (38.center) to (28);
		\draw [style=midarrow, bend left=270, looseness=2.00] (43.center) to (40.center);
		\draw [style=midarrow, in=0, out=150] (20) to (44.center);
		\draw (44.center) to (39.center);
	\end{pgfonlayer}
\end{tikzpicture}
    \quad $=$ \quad
    \begin{tikzpicture}[scale=2]
	\begin{pgfonlayer}{nodelayer}
		\node [style=none, label={left:$S\phantom{'^*}$}] (10) at (-1.5, -1.5) {};
		\node [style=none, label={left:$S'^*$}] (12) at (-1.5, -1) {};
		\node [style=morphism] (20) at (0, 0.75) {$l^*$};
		\node [style=morphism] (28) at (0, -0.1) {$r$};
		\node [style=none] (31) at (-1, -1) {};
		\node [style=none] (34) at (1, -1.5) {};
		\node [style=none] (35) at (1, -1) {};
		\node [style=none] (36) at (1, -0.5) {};
		\node [style=none] (37) at (1, 1.25) {};
		\node [style=none, label={left:$A'\phantom{^*}$}] (38) at (-1.5, -0.5) {};
		\node [style=none, label={left:$A^*\phantom{'}$}] (39) at (-1.5, 1.25) {};
		\node [style=none] (44) at (-1, 1.25) {};
	\end{pgfonlayer}
	\begin{pgfonlayer}{edgelayer}
		\draw (31.center) to (12.center);
		\draw [style=midarrow, bend right=90, looseness=2.00] (34.center) to (37.center);
		\draw [style=midarrow, bend left=90, looseness=2.00] (36.center) to (35.center);
		\draw [style=midarrow] (10.center) to (34.center);
		\draw [style=midarrow] (35.center) to (31.center);
		\draw [style=midarrow, bend right=90, looseness=2.25] (20) to (28);
		\draw [style=midarrow, in=-330, out=-180] (37.center) to (20);
		\draw [in=-180, out=-30] (28) to (36.center);
		\draw [style=midarrow, in=-165, out=0] (38.center) to (28);
		\draw [style=midarrow, in=0, out=150] (20) to (44.center);
		\draw (44.center) to (39.center);
		\draw [style=midarrow, bend right=90, looseness=2.25] (28) to (20);
	\end{pgfonlayer}
\end{tikzpicture}
  \end{center}
  The unit is similar.
\end{proof}

\section{Free Compact Closure}

We conclude by showing that this $\Atemp_\C$ is the free compact
closed category on a symmetric monoidal category $\C$.

There is an existing construction of the free compact closed category
given by composing three simpler constructions: \cite{ksw:feedback}
constructs the free \emph{feedback} monoidal category on a monoidal
category before quotienting it to form the free \emph{traced} monoidal
category. Then \cite{jsv:traced} constructs the free compact closed
category on a traced monoidal category. In the literature one also
finds the free compact closed category on a bare
category~\cite{kl:coherence-compact-closed}, and the free compact
closed category on a \emph{closed} monoidal category
\cite{day:note-compact-closed}, but these are less comparable to the
construction given here.


As a first step, any functor from $\C$ into a compact closed category
factors through $\Atemp_\C$.

\begin{theorem}
  Suppose $\C$ is a symmetric monoidal category and $\D$ is a compact
  closed category. For any symmetric monoidal functor $F : \C \to \D$,
  there is a factorisation
  \begin{center}
    \begin{tikzcd}
      \C \ar[dr, "\iota" swap] \ar[rr, "F"] & & \D \\
      & \Atemp_\C \ar[ur, "{\hat{F}}" swap]
    \end{tikzcd}
  \end{center}
\end{theorem}

Here we follow the shape of the argument in \cite[Section
5]{jsv:traced}.

\begin{proof}
  Any object $(A, A')$ of $\Atemp_\C$ is isomorphic to
  \begin{align*}
    (A, A') \cong (A', I)^* \otimes (A, I) = (\iota A')^* \otimes \iota A
  \end{align*}
  which forces us to define $\hat{F}(A, A') := (FA')^* \otimes FA$. On
  a morphism $(l \mid r)$, define $\hat{F}(l \mid r)$ as indicated in
  the diagram
  \begin{center}
    \begin{tikzpicture}[scale=2]
	\begin{pgfonlayer}{nodelayer}
		\node [style=none, label={left:$FS$}] (10) at (-2.25, -1.25) {};
		\node [style=none, label={left:$FS'^*$}] (12) at (-2.25, 1) {};
		\node [style=morphism] (20) at (0, -0.75) {$Fl$};
		\node [style=morphism] (28) at (0, 0.5) {$Fr^*$};
		\node [style=none] (30) at (-1, -1.25) {};
		\node [style=none] (31) at (-1, 1) {};
		\node [style=none] (34) at (1, -1.25) {};
		\node [style=none] (35) at (1, 1) {};
		\node [style=none, label={right:$FA'^*$}] (42) at (2.25, 1) {};
		\node [style=none, label={right:$FA$}] (44) at (2.25, -1.25) {};
	\end{pgfonlayer}
	\begin{pgfonlayer}{edgelayer}
		\draw [style=midarrow] (10.center) to (30.center);
		\draw [style=midarrow] (31.center) to (12.center);
		\draw [in=-180, out=30] (28) to (35.center);
		\draw [in=-150, out=0] (30.center) to (20);
		\draw [in=180, out=-30] (20) to (34.center);
		\draw [in=0, out=150] (28) to (31.center);
		\draw [style=midarrow] (42.center) to (35.center);
		\draw [style=midarrow] (34.center) to (44.center);
		\draw [style=midarrow, in=360, out=0, looseness=2.25] (20) to (28);
		\draw [style=midarrow, in=-180, out=180, looseness=2.25] (28) to (20);
	\end{pgfonlayer}
\end{tikzpicture}
  \end{center}
  This is well defined with respect to the quotient on $\Learn_\C$,
  because the snake equations already hold in $\D$. The functor
  $\hat{F}$ behaves nicely with the inclusion
  $\iota : \C \times \C^\op \to \Atemp_\C$, because
  $\hat{F}(\iota(f, g)) = g^* \otimes f$ for any two morphisms of
  $\C$.

  There are now several things to check. First, that $\hat{F}$ is a
  functor. Preservation of identity follows because the cup and cap
  for the unit object in $\D$ are the unitors.
  Preservation of composition is the equivalence of the following
  diagrams:
  \begin{center}
    \begin{tikzpicture}[scale=2]
	\begin{pgfonlayer}{nodelayer}
		\node [style=none, label={left:$FA$}] (10) at (-3.75, -1.25) {};
		\node [style=none, label={left:$FA'^*$}] (12) at (-3.75, 1) {};
		\node [style=morphism] (20) at (-1.5, -0.75) {$Fl_1$};
		\node [style=morphism] (28) at (-1.5, 0.5) {$Fr_1{}^*$};
		\node [style=none] (30) at (-2.5, -1.25) {};
		\node [style=none] (31) at (-2.5, 1) {};
		\node [style=none] (34) at (-0.5, -1.25) {};
		\node [style=none] (35) at (-0.5, 1) {};
		\node [style=morphism] (47) at (1.75, -0.75) {$Fl_2$};
		\node [style=morphism] (48) at (1.75, 0.5) {$Fr_2{}^*$};
		\node [style=none] (49) at (0.75, -1.25) {};
		\node [style=none] (50) at (0.75, 1) {};
		\node [style=none] (51) at (2.75, -1.25) {};
		\node [style=none] (52) at (2.75, 1) {};
		\node [style=none, label={right:$FC'^*$}] (53) at (4, 1) {};
		\node [style=none, label={right:$FC$}] (54) at (4, -1.25) {};
	\end{pgfonlayer}
	\begin{pgfonlayer}{edgelayer}
		\draw [style=midarrow] (10.center) to (30.center);
		\draw [style=midarrow] (31.center) to (12.center);
		\draw [in=-180, out=30] (28) to (35.center);
		\draw [in=-150, out=0] (30.center) to (20);
		\draw [in=180, out=-30] (20) to (34.center);
		\draw [in=0, out=150] (28) to (31.center);
		\draw [style=midarrow, in=360, out=0, looseness=2.25] (20) to (28);
		\draw [style=midarrow, in=-180, out=180, looseness=2.25] (28) to (20);
		\draw [in=-180, out=30] (48) to (52.center);
		\draw [in=-150, out=0] (49.center) to (47);
		\draw [in=180, out=-30] (47) to (51.center);
		\draw [in=0, out=150] (48) to (50.center);
		\draw [style=midarrow] (53.center) to (52.center);
		\draw [style=midarrow] (51.center) to (54.center);
		\draw [style=midarrow, in=360, out=0, looseness=2.25] (47) to (48);
		\draw [style=midarrow, in=-180, out=180, looseness=2.25] (48) to (47);
		\draw [style=midarrow] (50.center) to (35.center);
		\draw [style=midarrow] (34.center) to (49.center);
	\end{pgfonlayer}
\end{tikzpicture}
    \quad $=$ \quad
    \begin{tikzpicture}[scale=2]
	\begin{pgfonlayer}{nodelayer}
		\node [style=none, label={left:$FA$}] (10) at (-3.5, -2) {};
		\node [style=none] (30) at (-2.25, -2) {};
		\node [style=none] (51) at (2.25, -2) {};
		\node [style=none, label={right:$FC$}] (54) at (3.5, -2) {};
		\node [style=none] (55) at (-0.5, -0.5) {};
		\node [style=none] (56) at (0.5, -0.5) {};
		\node [style=none] (57) at (-0.5, -1) {};
		\node [style=none] (59) at (-0.5, -2) {};
		\node [style=morphism] (60) at (-1.25, -1.5) {$Fl_1$};
		\node [style=none] (61) at (-0.5, -1) {};
		\node [style=none] (64) at (-2, -0.5) {};
		\node [style=none] (65) at (2, -0.5) {};
		\node [style=none] (67) at (0.5, -2) {};
		\node [style=none] (69) at (0.5, -1) {};
		\node [style=morphism] (85) at (1.25, -1.5) {$Fl_2$};
		\node [style=none, label={right:$FC'^*$}] (87) at (3.5, 1.75) {};
		\node [style=none] (88) at (2.25, 1.75) {};
		\node [style=none] (89) at (-2.25, 1.75) {};
		\node [style=none, label={left:$FA'^*$}] (90) at (-3.5, 1.75) {};
		\node [style=none] (91) at (0.5, 0.25) {};
		\node [style=none] (92) at (-0.5, 0.25) {};
		\node [style=none] (94) at (0.5, 1.75) {};
		\node [style=morphism] (95) at (1.25, 1.25) {$Fr_2{}^*$};
		\node [style=none] (96) at (0.5, 0.75) {};
		\node [style=none] (97) at (2, 0.25) {};
		\node [style=none] (98) at (-2, 0.25) {};
		\node [style=none] (99) at (-0.5, 1.75) {};
		\node [style=none] (100) at (-0.5, 0.75) {};
		\node [style=morphism] (101) at (-1.25, 1.25) {$Fr_1{}^*$};
	\end{pgfonlayer}
	\begin{pgfonlayer}{edgelayer}
		\draw [style=midarrow] (10.center) to (30.center);
		\draw [style=midarrow] (51.center) to (54.center);
		\draw [in=-180, out=0, looseness=1.25] (57.center) to (56.center);
		\draw [in=-180, out=-60, looseness=0.75] (60) to (59.center);
		\draw [in=-180, out=60, looseness=0.75] (60) to (61.center);
		\draw [in=-150, out=0] (30.center) to (60);
		\draw (64.center) to (55.center);
		\draw (56.center) to (65.center);
		\draw [style=midarrow] (59.center) to (67.center);
		\draw [in=0, out=180, looseness=1.25] (69.center) to (55.center);
		\draw [in=0, out=-120, looseness=0.75] (85) to (67.center);
		\draw [in=0, out=120, looseness=0.75] (85) to (69.center);
		\draw [in=180, out=-30] (85) to (51.center);
		\draw [style=midarrow] (87.center) to (88.center);
		\draw [style=midarrow] (89.center) to (90.center);
		\draw [in=0, out=120, looseness=0.75] (95) to (94.center);
		\draw [in=0, out=-120, looseness=0.75] (95) to (96.center);
		\draw [in=30, out=180] (88.center) to (95);
		\draw (97.center) to (91.center);
		\draw (92.center) to (98.center);
		\draw [style=midarrow] (94.center) to (99.center);
		\draw [in=180, out=0, looseness=1.25] (100.center) to (91.center);
		\draw [in=180, out=60, looseness=0.75] (101) to (99.center);
		\draw [in=180, out=-60, looseness=0.75] (101) to (100.center);
		\draw [in=0, out=150] (101) to (89.center);
		\draw [in=180, out=-60, looseness=0.75] (101) to (100.center);
		\draw [style=midarrow, in=-180, out=180, looseness=1.75] (98.center) to (64.center);
		\draw [style=midarrow, in=0, out=0, looseness=1.75] (65.center) to (97.center);
		\draw [style=midarrow, in=-180, out=-180, looseness=1.50] (101) to (60);
		\draw [in=0, out=-120, looseness=0.75] (95) to (96.center);
		\draw [in=-180, out=30] (95) to (88.center);
		\draw [in=0, out=180, looseness=1.25] (96.center) to (92.center);
		\draw [style=midarrow, in=0, out=0, looseness=1.50] (85) to (95);
	\end{pgfonlayer}
\end{tikzpicture}
  \end{center}

  We must equip $\hat{F}$ with the structure of a monoidal functor.
  For the structure isomorphism
  $\phi : \hat{F}(A_L, A'_L) \otimes \hat{F}(A_R, A'_R) \to
  \hat{F}(A_L \otimes A_R, A'_R \otimes A'_L)$, we take the composite
  \begin{align*}
    &\hat{F}(A_L, A'_L) \otimes \hat{F}(A_R, A'_R) \\
    &= (FA_L \otimes FA'_L{}^*) \otimes (FA_R \otimes FA'_R{}^*) \\
    &\to (FA_L \otimes FA_R) \otimes (FA'_L{}^*  \otimes FA'_R{}^*) \\
    &\to (FA_L \otimes FA_R) \otimes (FA'_R  \otimes FA'_L)^* \\
    &\to F(A_L \otimes FA_R) \otimes (F(A'_R  \otimes A'_L))^* \\
    &= \hat{F}(A_L \otimes A_R, A'_R \otimes A'_L) \\
    &= \hat{F}((A_L, A'_L) \otimes (A_R, A'_R))
  \end{align*}
  That this gives a monoidal functor follows from the equivalence of
  the diagrams
  
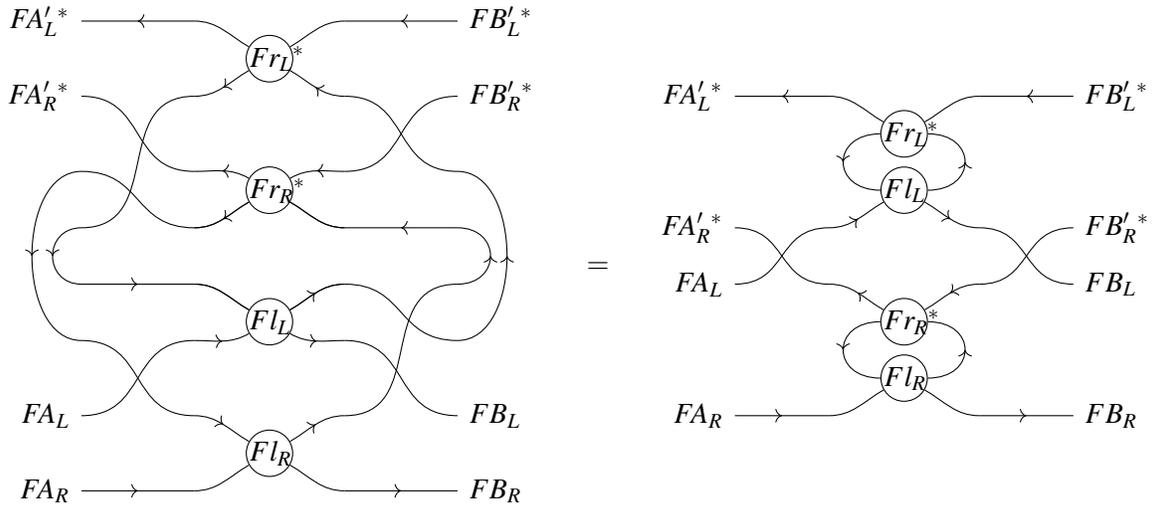
\begin{center}
    \begin{tikzpicture}[scale=2]
	\begin{pgfonlayer}{nodelayer}
		\node [style=none, label={right:$FB_R$}] (42) at (2.5, -3) {};
		\node [style=none] (45) at (-2.5, -1) {};
		\node [style=none] (46) at (-2.5, -0.25) {};
		\node [style=none, label={left:$FA_R$}] (47) at (-2.5, -3) {};
		\node [style=none, label={left:$FA_L$}] (48) at (-2.5, -2) {};
		\node [style=none] (49) at (-1, -3) {};
		\node [style=none] (50) at (1, -3) {};
		\node [style=none] (51) at (2.5, -1) {};
		\node [style=none] (52) at (2.5, -0.25) {};
		\node [style=none] (54) at (-2.5, -1) {};
		\node [style=none] (55) at (-1, -1) {};
		\node [style=none] (56) at (-2.5, -2) {};
		\node [style=morphism] (57) at (0, -2.5) {$Fl_R$};
		\node [style=none] (58) at (-1, -2) {};
		\node [style=none] (59) at (-1, -0.25) {};
		\node [style=morphism] (60) at (0, -0.75) {$Fl_L$};
		\node [style=none] (61) at (1, -1) {};
		\node [style=none] (62) at (1, -2) {};
		\node [style=none, label={right:$FB_L$}] (63) at (2.5, -2) {};
		\node [style=none] (64) at (1, -0.25) {};
		\node [style=none, label={left:$FA'_L{}^*$}] (65) at (-2.5, 3.25) {};
		\node [style=none] (66) at (2.5, 1.25) {};
		\node [style=none] (67) at (2.5, 0.5) {};
		\node [style=none, label={right:$FB'_L{}^*$}] (68) at (2.5, 3.25) {};
		\node [style=none, label={right:$FB'_R{}^*$}] (69) at (2.5, 2.25) {};
		\node [style=none] (70) at (1, 3.25) {};
		\node [style=none] (71) at (-1, 3.25) {};
		\node [style=none] (72) at (-2.5, 1.25) {};
		\node [style=none] (73) at (-2.5, 0.5) {};
		\node [style=none] (74) at (2.5, 1.25) {};
		\node [style=none] (75) at (1, 1.25) {};
		\node [style=none] (76) at (2.5, 2.25) {};
		\node [style=morphism] (77) at (0, 2.75) {$Fr_L{}^*$};
		\node [style=none] (78) at (1, 2.25) {};
		\node [style=none] (79) at (1, 0.5) {};
		\node [style=morphism] (80) at (0, 1) {$Fr_R{}^*$};
		\node [style=none] (81) at (-1, 1.25) {};
		\node [style=none] (82) at (-1, 2.25) {};
		\node [style=none, label={left:$FA'_R{}^*$}] (83) at (-2.5, 2.25) {};
		\node [style=none] (84) at (-1, 0.5) {};
	\end{pgfonlayer}
	\begin{pgfonlayer}{edgelayer}
		\draw [style=midarrow] (47.center) to (49.center);
		\draw [in=-180, out=0, looseness=1.25] (56.center) to (55.center);
		\draw [in=180, out=-30] (57) to (50.center);
		\draw [in=0, out=-150] (57) to (49.center);
		\draw [in=0, out=180, looseness=1.25] (58.center) to (54.center);
		\draw [in=0, out=150] (60) to (59.center);
		\draw [style=midarrow, in=-150, out=0] (55.center) to (60);
		\draw [style=midarrow] (46.center) to (59.center);
		\draw [in=150, out=0] (59.center) to (60);
		\draw [in=0, out=180, looseness=1.25] (63.center) to (61.center);
		\draw [style=midarrow, in=-180, out=30] (57) to (62.center);
		\draw [in=30, out=180] (64.center) to (60);
		\draw [style=midarrow, in=-180, out=30] (60) to (64.center);
		\draw [in=-180, out=0, looseness=1.25] (62.center) to (52.center);
		\draw [in=0, out=180, looseness=1.25] (51.center) to (64.center);
		\draw [style=midarrow] (50.center) to (42.center);
		\draw [style=midarrow] (68.center) to (70.center);
		\draw [in=0, out=-180, looseness=1.25] (76.center) to (75.center);
		\draw [in=0, out=150] (77) to (71.center);
		\draw [in=-180, out=30] (77) to (70.center);
		\draw [in=-180, out=0, looseness=1.25] (78.center) to (74.center);
		\draw [in=-180, out=-30] (80) to (79.center);
		\draw [style=midarrow, in=30, out=-180] (75.center) to (80);
		\draw [style=midarrow] (67.center) to (79.center);
		\draw [in=-30, out=-180] (79.center) to (80);
		\draw [in=-180, out=0, looseness=1.25] (83.center) to (81.center);
		\draw [style=midarrow, in=0, out=-150] (77) to (82.center);
		\draw [in=-150, out=0] (84.center) to (80);
		\draw [style=midarrow, in=0, out=-150] (80) to (84.center);
		\draw [in=0, out=-180, looseness=1.25] (82.center) to (73.center);
		\draw [in=-180, out=0, looseness=1.25] (72.center) to (84.center);
		\draw [style=midarrow] (71.center) to (65.center);
		\draw [style=midarrow, in=-180, out=180, looseness=1.75] (73.center) to (46.center);
		\draw [style=midarrow, in=-180, out=-180] (72.center) to (54.center);
		\draw [style=midarrow, in=0, out=0, looseness=2.00] (52.center) to (67.center);
		\draw [style=midarrow, in=0, out=0] (51.center) to (74.center);
		\draw [style=midarrow, in=0, out=150] (80) to (81.center);
		\draw [style=midarrow, in=180, out=-30] (60) to (61.center);
		\draw [style=midarrow, in=150, out=0] (58.center) to (57);
		\draw [style=midarrow, in=330, out=180] (78.center) to (77);
	\end{pgfonlayer}
\end{tikzpicture}
    \quad $=$ \quad
    \begin{tikzpicture}[scale=2]
	\begin{pgfonlayer}{nodelayer}
		\node [style=none, label={left:$FA_L$}] (10) at (-2.25, -0.25) {};
		\node [style=none, label={left:$FA_L'{}^*$}] (12) at (-2.25, 2.25) {};
		\node [style=morphism] (20) at (0, 1) {$Fl_L$};
		\node [style=morphism] (28) at (0, 1.75) {$Fr_L{}^*$};
		\node [style=none] (31) at (-1, 2.25) {};
		\node [style=none] (35) at (1, 2.25) {};
		\node [style=none, label={right:$FB_L'{}^*$}] (42) at (2.25, 2.25) {};
		\node [style=none, label={right:$FB_L$}] (44) at (2.25, -0.25) {};
		\node [style=none, label={left:$FA_R$}] (45) at (-2.25, -2) {};
		\node [style=none, label={left:$FA_R'{}^*$}] (46) at (-2.25, 0.5) {};
		\node [style=morphism] (47) at (0, -1.5) {$Fl_R$};
		\node [style=none] (49) at (-1, -2) {};
		\node [style=none] (51) at (1, -2) {};
		\node [style=none, label={right:$FB_R'{}^*$}] (53) at (2.25, 0.5) {};
		\node [style=none, label={right:$FB_R$}] (54) at (2.25, -2) {};
		\node [style=none] (65) at (-1, 0.5) {};
		\node [style=none] (66) at (-1, -0.25) {};
		\node [style=none] (69) at (1, 0.5) {};
		\node [style=morphism] (71) at (0, -0.75) {$Fr_R{}^*$};
		\node [style=none] (72) at (1, -0.25) {};
	\end{pgfonlayer}
	\begin{pgfonlayer}{edgelayer}
		\draw [style=midarrow] (31.center) to (12.center);
		\draw [in=-180, out=30] (28) to (35.center);
		\draw [in=0, out=150] (28) to (31.center);
		\draw [style=midarrow] (42.center) to (35.center);
		\draw [style=midarrow, in=360, out=0, looseness=2.25] (20) to (28);
		\draw [style=midarrow, in=-180, out=180, looseness=2.25] (28) to (20);
		\draw [style=midarrow] (45.center) to (49.center);
		\draw [in=-150, out=0] (49.center) to (47);
		\draw [in=180, out=-30] (47) to (51.center);
		\draw [style=midarrow] (51.center) to (54.center);
		\draw [style=midarrow, in=-150, out=0] (65.center) to (20);
		\draw [in=-180, out=0, looseness=1.25] (10.center) to (65.center);
		\draw [in=0, out=180, looseness=1.25] (66.center) to (46.center);
		\draw [style=midarrow, in=180, out=-30] (20) to (69.center);
		\draw [in=0, out=180, looseness=1.25] (44.center) to (69.center);
		\draw [style=midarrow, in=-180, out=180, looseness=2.25] (71) to (47);
		\draw [style=midarrow, in=360, out=0, looseness=2.25] (47) to (71);
		\draw [style=midarrow, in=0, out=150] (71) to (66.center);
		\draw [in=-180, out=0, looseness=1.25] (72.center) to (53.center);
		\draw [style=midarrow, in=30, out=-180] (72.center) to (71);
	\end{pgfonlayer}
\end{tikzpicture}
  \end{center}
  The functor is also symmetric monoidal, as the symmetry morphisms of
  $\Atemp_\C$ are inherited from the symmetry morphisms of $\C \times
  \C^\op$ via $\iota$.
\end{proof}

Let $\SymMon_g$ denote the 2-category of symmetric monoidal
categories, monoidal functors and monoidal natural
\emph{isomorphisms}, and let $\Comp$ denote the 2-category of compact
closed categories, monoidal functors and monoidal natural
transformations. Recall that no condition on monoidal functors between
compact closed categories is necessary: duals are always preserved up
to canonical isomorphism. We also do not need to explicitly restrict
the 2-cells in $\Comp$ to be invertible, because \emph{any} such
monoidal natural transformation is invertible:

\begin{proposition}{\cite[Proposition 7.1]{js:braided}}\label{prop:nat-inverse}
  Any monoidal natural transformation $\alpha : F \to G$ in $\Comp$
  has an inverse given by
  \begin{align*}
    G A \to (G (A^*))^* \xrightarrow{(\alpha_{A^*})^*} (F (A^*))^* \to F A
  \end{align*}
  at each object $A$. \qed
\end{proposition}

\begin{theorem}
  $\Atemp_\C$ is the free compact closed category on a symmetric
  monoidal category $\C$. That is, $\Atemp : \SymMon_g \to \Comp$
  assembles into a 2-functor that is left biadjoint to the inclusion
  $\Comp \to \SymMon_g$, with the unit of the biadjunction having
  component $\iota : \C \to \Atemp_\C$ at $\C$.
\end{theorem}
\begin{proof}
  Suppose $\D$ is compact closed. For two symmetric monoidal functors
  $F, G : \Atemp_\C \to \D$, the restriction map along $\iota_\C$
  gives a map
  \begin{align*}
    \Comp(\Atemp_\C, \D)(F, G) \to \SymMon_g(\C, \D)(F \circ \iota, G \circ \iota)
  \end{align*}
  We claim that this restriction map is a bijection.

  Given a monoidal isomorphism
  $\beta : F \circ \iota \to G \circ \iota$, define the monoidal
  isomorphism $\overline{\beta} : F \to G$ whose component at
  $(A, A')$ is
  \begin{align*}
    F(A, A')
    &= F(\iota A \otimes (\iota A')^*)
    \to F\iota A \otimes F(\iota A'^*) \\
    &\to F\iota A \otimes (F\iota A')^*
    \xrightarrow{\beta_A \otimes (\beta^{-1}_{A'})^*} G\iota A
      \otimes (G\iota A')^* \\
    &\to G\iota A \otimes G(\iota A'^*)
    \to G(\iota A \otimes \iota A'^*)
    = G(A, A')
  \end{align*}
  Eliding the structural morphisms, this corresponds simply to the
  following diagram in $\D$:
  \begin{center}
    \begin{tikzpicture}[scale=2]
	\begin{pgfonlayer}{nodelayer}
		\node [style=none, label={left:$F A$}] (10) at (-1, -0.5) {};
		\node [style=none, label={left:$F A'^*$}] (12) at (-1, 0.5) {};
		\node [style=none, label={right:$G A'^*$}] (42) at (1, 0.5) {};
		\node [style=none, label={right:$G A$}] (44) at (1, -0.5) {};
		\node [style=morphism] (47) at (0, -0.5) {$\beta_A$};
		\node [style=morphism] (48) at (0, 0.5) {$\beta_{A'}^{-1}{}^*$};
	\end{pgfonlayer}
	\begin{pgfonlayer}{edgelayer}
		\draw [style=midarrow] (10.center) to (47);
		\draw [style=midarrow] (48) to (12.center);
		\draw [style=midarrow] (42.center) to (48);
		\draw [style=midarrow] (47) to (44.center);
	\end{pgfonlayer}
\end{tikzpicture}
  \end{center}
  Monoidalness of $\overline{\beta}$ follows immediately from the
  monoidalness of $\beta$.  For naturality, suppose we have a morphism
  $(l \mid r) : (A, A') \to (B, B')$ in $\Atemp_\C$. Naturality of
  $\overline{\beta}$ with respect to $(l \mid r)$ corresponds to the
  equivalence of the following diagrams:
  \begin{center}
    \begin{tikzpicture}[scale=2]
	\begin{pgfonlayer}{nodelayer}
		\node [style=none, label={left:$F A$}] (10) at (-1.75, -1.5) {};
		\node [style=none, label={left:$F A'^*$}] (12) at (-1.75, 1.25) {};
		\node [style=morphism] (20) at (0.5, -0.75) {$G l$};
		\node [style=morphism] (28) at (0.5, 0.5) {$G r^*$};
		\node [style=none] (34) at (1.5, -1.5) {};
		\node [style=none] (35) at (1.5, 1.25) {};
		\node [style=none, label={right:$G B'^*$}] (42) at (2, 1.25) {};
		\node [style=none, label={right:$G B$}] (44) at (2, -1.5) {};
		\node [style=morphism] (45) at (-0.75, -0.75) {$\beta_P$};
		\node [style=morphism] (46) at (-0.75, 0.5) {$\beta_P^{-1}{}^*$};
		\node [style=morphism] (47) at (-0.75, -1.5) {$\beta_A$};
		\node [style=morphism] (48) at (-0.75, 1.25) {$\beta_{A'}^{-1}{}^*$};
	\end{pgfonlayer}
	\begin{pgfonlayer}{edgelayer}
		\draw [in=-180, out=30] (28) to (35.center);
		\draw [in=-180, out=-30] (20) to (34.center);
		\draw [style=midarrow] (42.center) to (35.center);
		\draw [style=midarrow] (34.center) to (44.center);
		\draw [style=midarrow, in=360, out=0, looseness=2.25] (20) to (28);
		\draw (45) to (20);
		\draw (28) to (46);
		\draw [in=0, out=150] (28) to (48);
		\draw [in=210, out=0] (47) to (20);
		\draw [style=midarrow] (10.center) to (47);
		\draw [style=midarrow] (48) to (12.center);
		\draw [style=midarrow, in=180, out=180, looseness=2.25] (46) to (45);
	\end{pgfonlayer}
\end{tikzpicture}
    \quad $=$ \quad
    \begin{tikzpicture}[scale=2]
	\begin{pgfonlayer}{nodelayer}
		\node [style=none, label={right:$G B$}] (0) at (1.75, -1.5) {};
		\node [style=none, label={right:$G B'^*$}] (1) at (1.75, 1.25) {};
		\node [style=morphism] (2) at (-0.5, -0.75) {$F l$};
		\node [style=morphism] (3) at (-0.5, 0.5) {$F r^*$};
		\node [style=none] (4) at (-1.5, -1.5) {};
		\node [style=none] (5) at (-1.5, 1.25) {};
		\node [style=none, label={left:$F A'^*$}] (6) at (-2, 1.25) {};
		\node [style=none, label={left:$F A$}] (7) at (-2, -1.5) {};
		\node [style=morphism] (8) at (0.75, -0.75) {$\beta_Q$};
		\node [style=morphism] (9) at (0.75, 0.5) {$\beta_Q^{-1}{}^*$};
		\node [style=morphism] (10) at (0.75, -1.5) {$\beta_B$};
		\node [style=morphism] (11) at (0.75, 1.25) {$\beta_{B'}^{-1}{}^*$};
	\end{pgfonlayer}
	\begin{pgfonlayer}{edgelayer}
		\draw [in=0, out=150] (3) to (5.center);
		\draw [in=0, out=-150] (2) to (4.center);
		\draw [style=midarrow, in=-180, out=180, looseness=2.25] (2) to (3);
		\draw (8) to (2);
		\draw (3) to (9);
		\draw [in=180, out=30] (3) to (11);
		\draw [in=-30, out=180] (10) to (2);
		\draw [style=midarrow, in=0, out=0, looseness=2.25] (9) to (8);
		\draw [style=midarrow] (7.center) to (4.center);
		\draw [style=midarrow] (10) to (0.center);
		\draw [style=midarrow] (1.center) to (11);
		\draw [style=midarrow] (5.center) to (6.center);
	\end{pgfonlayer}
\end{tikzpicture}
  \end{center}
  using monoidalness and naturality of $\beta$ to pass it through $l$
  and $r$. The extraneous $\beta_P$ and $\beta_Q$ are cancelled with
  their inverses by passing them around the cup and cap.

  Finally, we must show that this process is inverse to restriction
  along $\iota$. It is clear from the definition that
  $\overline{\beta}\iota = \beta$. The interesting case is the
  converse, that $\overline{\alpha\iota} = \alpha$ for a monoidal
  isomorphism $\alpha : F \to G$. We verify
  \begin{align*}
    &F(A, A') \xrightarrow{\overline{\alpha\iota}_{(A, A')}} G(A, A') \\
    &= F(A, A') \xrightarrow{\sim} F\iota A \otimes (F\iota A')^*
      \xrightarrow{\alpha\iota_{A} \otimes (\alpha\iota_{A'}{}^{-1})^*} G\iota A \otimes
      (G\iota A')^* \xrightarrow{\sim} G(A, A')\\
    &= F(A, A') \xrightarrow{\sim} F\iota A \otimes (F\iota A')^*
      \xrightarrow{\alpha_{\iota A} \otimes (\alpha_{\iota A'}{}^{-1})^*} G\iota A \otimes
      (G\iota A')^* \xrightarrow{\sim} G(A, A')\\
    &= F(A, A') \xrightarrow{\sim} F\iota A \otimes (F\iota A')^*
      \xrightarrow{\sim} F\iota A \otimes (F(\iota A')^*)^{**}
      \xrightarrow{\alpha_{\iota A} \otimes (\alpha_{(\iota
      A')^*})^{**}}G\iota A \otimes (G(\iota A')^*)^{**} \\
      &\qquad \xrightarrow{\sim} G\iota A \otimes (G\iota A')^*
      \xrightarrow{\sim} G(A, A') \\
    &= F(A, A') \xrightarrow{\sim} F\iota A \otimes (F(\iota A')^*)
      \xrightarrow{\alpha_{\iota A} \otimes (\alpha_{(\iota A')^*})} G\iota A \otimes
      (G\iota A')^* \xrightarrow{\sim} G(A, A')\\
    &= F(A, A') \xrightarrow{\alpha_{(A, A')}} G(A, A')
  \end{align*}
  by expanding the explicit description of the inverse of
  $\alpha_{\iota A'}$ given by Proposition~\ref{prop:nat-inverse}, and using
  the monoidalness of $\alpha$ at the last step.

  From this we conclude that restriction along the inclusion $\iota$ is an
  equivalence of categories between $\Comp(\Atemp_\C, \D)$ and
  $\SymMon_g(\C, \D)$, making $\iota$ a \emph{biuniversal arrow}. It
  follows abstractly from~\cite[Theorem 9.17]{fiore:biadjoints} that
  $\Atemp$ assembles into a bifunctor $\Atemp : \SymMon_g \to \Comp$,
  left biadjoint to the inclusion.
\end{proof}

\begin{remark}
  The restriction to $\SymMon_g$ with natural \emph{isomorphisms} as 2-cells
  is necessary. As pointed out by~\cite{hk:correction}, this
  assumption is erroneously missing from the similar theorem given
  in~\cite{jsv:traced}.
\end{remark}


\bibliographystyle{eptcs}
\bibliography{optics-bibtex}
\end{document}